\documentclass[11pt]{article}%
\usepackage{graphicx}
\usepackage{float} 
\usepackage{fullpage}
\usepackage{amsfonts}
\usepackage{amsmath}
\usepackage{amssymb}
\usepackage{amsbsy}
\usepackage{amsthm}
\usepackage{epsfig}
\usepackage{graphicx}
\usepackage{colordvi}
\usepackage{xcolor}
\usepackage{graphics}
\usepackage{color}
\usepackage{bm}
\usepackage{verbatim} 
\numberwithin{equation}{section}
\usepackage{multirow}
\usepackage{booktabs}
\usepackage{enumerate}
\usepackage{mathtools}

\usepackage{url}

\usepackage{amssymb, epsfig,amssymb, latexsym}
\usepackage{amsfonts,psfrag,amsmath,bbm,color,cancel,epstopdf,multirow,algorithm,url,caption,threeparttable,ragged2e,wrapfig,overpic}

\newtheorem{theorem}{Theorem}[section]
\newtheorem{thm}{Theorem}[section]

\newtheorem{lemma}[theorem]{Lemma}
\newtheorem{dfn}[theorem]{Definition}

\newtheorem{alg}[theorem]{Algorithm}

\newtheorem{assumption}[theorem]{Assumption}

\newcommand{\goto}{\rightarrow}

\newcommand{\f}{\frac}

\DeclareMathOperator{\argmin}{argmin}

\newcommand{\nr}[1]{\ensuremath{\left\|{#1}\right\|}}
\newcommand{\nrs}[1]{\ensuremath{\|{#1}\|}} 

\newcommand{\xa}{x^\alpha}

\newcommand{\wa}{w^\alpha}

\begin{document}

\title{Acceleration of nonlinear solvers for natural convection problems}

\author{
Sara Pollock
\thanks{Department of Mathematics,  University of Florida, Gainesville, FL 32611 (s.pollock@ufl.edu).}
\and
Leo G. Rebholz
\thanks{School of Mathematical and Statistical Sciences, Clemson University, Clemson SC 29634 (rebholz@clemson.edu).}
\and
Mengying Xiao 
\thanks{Department of Mathematics,  University of Florida, Gainesville, FL 32611 (m.xiao@ufl.edu).}
}
\date{}

\maketitle

\begin{abstract}
This paper develops an efficient and robust solution technique 
for the steady Boussinesq model of non-isothermal flow using  
Anderson acceleration applied to a Picard iteration.
After analyzing the fixed point operator associated with the nonlinear iteration to 
prove that certain stability and regularity properties hold, we apply the authors' 
recently constructed theory for Anderson acceleration, which yields a convergence 
result for the Anderson accelerated Picard iteration for the Boussinesq system.  
The result shows that the leading term in the residual is  improved by the gain in the 
optimization problem, but at the cost of additional higher order terms that can be 
significant when the residual is large.  
We perform numerical tests that illustrate the theory, and show that a 2-stage choice 
of Anderson depth can be advantageous.  
We also consider Anderson acceleration applied to the Newton iteration for the 
Boussinesq equations, and observe that the acceleration allows the Newton iteration to 
converge for significantly higher Rayleigh numbers that it could without acceleration, 
even with a standard line search.
\end{abstract}

\section{Introduction}
Flows driven by natural convection (bouyancy) occur in many practical problems 
including ventilation, solar collectors, insulation in windows, cooling in electronics,
and many others \cite{CK11}.  
Such phenomena are typically modeled by
the Boussinesq system, which is given in a domain 
$\Omega\subset\mathbb{R}^d$ ($d$=2 or 3) by
\begin{align}\label{eqn:bouss}
u_t + (u\cdot \nabla ) u - \nu \Delta u + \nabla p 
&= Ri \langle 0, \theta \rangle^T + f,
\nonumber \\
\nabla \cdot u &= 0,
\nonumber \\
\theta_t + (u\cdot \nabla) \theta - \kappa \Delta \theta &= \gamma,
\end{align}
with $u$ representing the velocity field, $p$ the pressure, $\theta$ the temperature 
(or density), and with $f$ and $\gamma$ the external momentum forcing and thermal 
sources. The kinematic viscosity $\nu>0$ is defined as the inverse of the Reynolds 
number ($Re=\nu^{-1}$), and the thermal conductivity
$\kappa$ is given by
$\kappa=Re^{-1}Pr^{-1}$ where $Pr$ is the
Prandtl number and $Ri$ is the Richardson number accounting for the gravitational 
force.  Appropriate initial and boundary conditions are required to determine the 
system.  The Rayleigh number is defined by $Ra=Ri\cdot Re^2 \cdot Pr$, and higher $Ra$ 
leads to more complex physics as well as more difficulties in numerically solving the system.

Finding accurate solutions to the Boussinesq system requires the efficient solution
of a discretized nonlinear system based on \eqref{eqn:bouss}.  
We restrict our attention to those nonlinear systems arising 
from the steady Boussinesq system, since for nonlinear solvers, 
this is the more difficult case.  In time dependent problems, for instance, 
linearizations may be used which allow one to avoid solving nonlinear systems \cite{AKR17}.
If one does need to solve the nonlinear system at each time step (e.g. in cases where 
there are fast temporal dynamics), then our work below is relevant, and the 
analysis as it is developed here generally applies
as the time derivative term only improves the properties of the system. 
Moreover, at each time step one has access to good initial iterates to each subsequent
nonlinear problem (e.g., the solution at the last time step).

We will consider the Picard iteration for solving the steady Boussinesq system:  
Given initial $u_0$, define $u_k, \theta_k$, $k \ge 1$, by  
\begin{align}
(u_{k-1}\cdot \nabla ) u_k - \nu \Delta u_k + \nabla p_k & = Ri \langle 0,\theta_{k} \rangle^T + f, \label{it1a} \\
\nabla \cdot u_k &= 0, \label{it2a} \\
(u_{k-1}\cdot \nabla) \theta_k - \kappa \Delta \theta_k & = \gamma, \label{it3a}
\end{align}
with $u_k, \theta_k$ satisfying appropriate boundary conditions.  
Our analysis will consider this iteration together with a finite element 
discretization.  A critical feature of this iteration is that the momentum/mass 
equations are decoupled in each iteration from the energy equation.
At each Picard iteration, one first solves the linear system \eqref{it3a}, and then
solves the linear system \eqref{it1a}-\eqref{it2a}, which 
will be much more efficient than the 
Newton iteration (add $(u_k-u_{k-1})\cdot\nabla u_{k-1}$ and $(u_k-u_{k-1})\cdot\nabla \theta_{k-1}$ to the momentum and energy equations, respectively).  
The difficulty with the Newton iteration is that it is fully coupled: at each
iteration the linear solve is for $(u_k,p_k,\theta_k)$ together.
Such block linear systems can be
difficult to solve since little is known about how 
to effectively precondition them.  

The decoupled iteration \eqref{it1a}-\eqref{it3a}, on the other hand, requires solving an Oseen linear system and a temperature transport system; many methods exist
for effectively solving these linear systems \cite{bgl05,benzi,elman:silvester:wathen}.
However, even if each step of the decoupled iteration is fast, convergence properties 
for this iteration are not as good as those for the Newton iteration
(provided a good initial guess).
This motivates accelerating the nonlinear iteration for the decoupled scheme to achieve
a method where each linear solve is fast, and which produces a sequence of iterates that
converges rapidly to the solution.
The purpose of this paper is to consider the decoupled Boussinesq iteration \eqref{it1a}-\eqref{it3a} together with Anderson acceleration.

Anderson acceleration \cite{anderson65} is an extrapolation technique which, after
computing an update step from the current iterate, forms the next iterate from a
linear combination of $m$ previous iterates and update steps.  The parameter $m$
is referred to as the algorithmic depth of the iteration. 
The linear combination is chosen 
as the one which minimizes the norm of the $m$ most recent update steps.  
The technique has increased in popularity since its efficient implementation and 
use on a variety of applications was described in \cite{WaNi11}. Convergence theory
can be found in \cite{K18,ToKe15}, and its relation to quasi-Newton methods has been
developed in \cite{Eyert96,FaSa09}.
As shown in \cite{EPRX20,PRX19,PR19a}, as well as \cite{ToKe15}, the choice of norm 
used in the inner optimization problem plays an 
important role in the effectiveness of the acceleration. Local improvement in
the convergence rate of the iteration can be shown if the original fixed-point
iteration is contractive in a given norm.  However, as further discussed in 
\cite{PR19a}, the iteration does not have to be contractive for Anderson acceleration
to be effective.

Herein, we extend the theory of improved convergence using Anderson acceleration 
to the Boussinesq system with the iteration \eqref{it1a}-\eqref{it3a}, 
and demonstrate its efficiency on benchmark problems for natural convection.  
Our results show that applied in accordance with the developed theory of 
\cite{PR19a}, the acceleration has a substantial and positive impact on the efficiency 
of the iteration, and even provides convergence when the iteration would otherwise 
fail.  

Additionally, although the focus of the paper is for the Picard iteration 
\eqref{it1a}-\eqref{it3a}, numerical testing of Anderson acceleration for the related
 Newton iteration is also performed.  Anderson acceleration has been (numerically)  
shown on several test problems to enlarge the domain of convergence
for Newton iterations, although it can locally slow convergence, 
reducing the natural quadratic order of convergence to subquadratic, in the vicinity of a solution
\cite{EPRX20,PS20}.
Our results show that Anderson acceleration applied to the Newton iteration for 
the Boussinesq system substantially improves the performance of the solver at higher 
Rayleigh numbers.

The remainder of the paper is organized as follows: In section \ref{sec:notation} 
we provide some background on a stable finite element spatial discretization for the 
steady Boussinesq equations, and on Anderson acceleration applied to a general fixed 
point iteration. 
In section \ref{sec:decoupled} we give a decoupled fixed point iteration for the 
steady Boussinesq system under the aforementioned finite element framework, 
and show that this iteration is Lipschitz Fr\'echet differentiable and satisfies the 
assumptions of \cite{PR19a}. 
Subsection \ref{subsec:accelbouss} states the Anderson accelerated iteration 
specifically
for, and as it is applied to the steady Boussinesq system, and presents the convergence
results for this problem. 
In section \ref{sec:numerics} 
we report on results of a heated cavity problem 
with varying Rayleigh number for the decoupled iteration of the steady Boussinesq 
system, and show that Anderson acceleration can have a notable positive impact on the 
convergence speed, especially for problems featuring a large Rayleigh number. 
Section \ref{sec:newton} shows numerical results for Anderson acceleration applied
to the related Newton iteration.

\section{Notation and Mathematical Preliminaries}\label{sec:notation}

This section will provide notation, mathematical preliminaries and background, 
to allow for a smooth analysis in later sections.  First,
we will give function space and notational details, followed by finite element discretization preliminaries, and finally a brief review
of Anderson acceleration.

The domain $\Omega\subset\mathbb{R}^{d}$ is assumed to be simply connected   
and to either be a convex polytope or have a smooth boundary.
The $L^2(\Omega)$ norm
and inner product will be denoted by $\| \cdot \|$ and $(\cdot,\cdot)$, respectively, and all other norms will be labeled with subscripts. Common boundary conditions for velocity and temperature are the Dirichlet conditions given by\begin{equation}
\label{dbc}
u =g(x) \text{ on } \partial\Omega, \quad T = h(x) \text{ on }\partial\Omega,
\end{equation}
and mixed Dirichlet/Neumann conditions 
\begin{eqnarray} 
\label{mbc}
u=g(x) \text{ on }\partial\Omega, \quad  T = h(x) \text{ on }\Gamma_1, \quad \nabla 
T\cdot n = 0 \text{ on } \Omega\backslash \Gamma_1,
\end{eqnarray}
where $\Gamma_1\subset \partial\Omega,$ $|\Gamma_1|>0$, and $g(x), h(x)$ are given functions.
For simplicity, we consider the homogeneous case where $g(x)=0$ and $h(x) = 0,$ 
and note that our results are extendable to the non-homogeneous case. 

\subsection{Mathematical preliminaries}\label{subsec:mathprelim}
In this subsection, we consider the system \eqref{it1a}-\eqref{it3a}, coupled with 
the Dirichlet conditions \eqref{dbc}, and present some standard results that will be
used later.  These results additionally hold for system \eqref{it1a}-\eqref{it3a} 
with the mixed boundary conditions \eqref{mbc}, which can by seen by an integration by 
parts. 
The natural function spaces for velocity, pressure, and temperature are given by
\begin{align*}
X &= H_0^1(\Omega)^d := \{v\in L^2(\Omega)^d: \nabla v\in L^2(\Omega)^{d\times d}, v=0 \text{ on }\partial \Omega \}, \\
Q &= L_0^2(\Omega) := \{q\in L^2(\Omega): \int_\Omega q\ dx = 0 \},\\
W &= H_0^1(\Omega).
\end{align*}
The Poincar\'{e} inequality is known to hold in both $X$ and $W$
\cite{laytonbook}: there exists $C_P>0$ dependent only on the domain $\Omega$ 
satisfying $$\| v\| \le C_P\| \nabla v\|,$$ for any $v\in X$, or $v\in W$.

Define the  trilinear form: $b :  X \times X \times X \to \mathbb{R}$ such that for any $u,v,w\in X$
$$b(u,v,w) \coloneq  \frac12 ((u\cdot \nabla v, w) - (u\cdot \nabla w, v)). $$
The operator $b$ is skew-symmetric and satisfies 
\begin{align} 
b(u,v,v) & =  0 ,\label{ss} \\
\label{bbd}
b(u,v,w) &\le  M \| \nabla u\| \|\nabla v\| \|\nabla w\|,
\end{align} 
for any $u,v,w\in X$, with $M$ depending only on $\Omega$ \cite[Chapter 6]{laytonbook}.
 Similarly, 
define $b^*: X\times W \times W \to \mathbb{R}$ such that for any $v\in X, \phi,\psi\in W$ 
$$b^*(v,\phi,\psi) := \frac12( (v\cdot \nabla \phi, \psi) - (v\cdot\nabla \psi, \phi)).$$ One can easily check that $b^*$ also satisfies \eqref{ss}-\eqref{bbd}.

We will denote by $\tau_h$ a regular, conforming triangulation of $\Omega$ with maximum element diameter $h$.  The finite element spaces will be denoted as
$X_h \subset X, Q_h \subset Q, W_h \subset W$, and we require that the 
$(X_h,Q_h)$ pair satisfies the usual discrete inf-sup condition \cite{BS08}. 
Common choices are Taylor-Hood elements \cite{BS08}, 
Scott-Vogelius elements on an appropriate mesh 
\cite{arnold:qin:scott:vogelius:2D,MR2519595,zhang:scott:vogelius:3D}, 
or the mini element \cite{ABF84}. 

Define the discretely divergence-free subspace $V_h$ by
\begin{align}\label{eqn:ddfs}
V_h \coloneqq \{ v \in X_h,\ (\nabla \cdot v_h,q_h)=0 \ \forall q_h\in Q_h \}.
\end{align}
Utilizing the space $V_h$ will help simplify some of the analysis that follows.
The discrete stationary Boussinesq equations for $(u,\theta) \in (V_h, W_h)$ can now
be written in weak form as:
\begin{align}
\label{b1a}
b(u,u,v) + \nu(\nabla u, \nabla v) &= Ri (\langle 0, \theta\rangle^T, v) + (f,v) ,\\
\label{b2a}
b^*(u,\theta, \chi) + \kappa (\nabla \theta, \nabla \chi) & = (\gamma , \chi),
\end{align}
for any $(v,\chi)\in (V_h, W_h)$. One can easily check that 
\begin{equation}
\label{bd1}
\|\nabla u\| \le  K_1 : =RiC_P^2\nu^{-1}\kappa^{-1}\|\gamma\|_{-1}+ \nu^{-1}\|f\|_{-1} 
\text{  and  } \|\nabla \theta\| \le K_2:=\kappa^{-1} \|\gamma \|_{-1},
\end{equation} 
by choosing $v = u, \chi = \theta$ in \eqref{b1a}-\eqref{b2a}. 
We next present a set of sufficient conditions to guarantee the uniqueness of the 
solution for the discrete stationary Boussinesq equations.
\begin{lemma}[Small data condition]
\label{sdc}
The following are sufficient conditions for the  system \eqref{b1a}-\eqref{b2a} 
to have a unique solution
\begin{eqnarray}
\label{sdcineq1}
\nu^{-1}M\left( 2K_1 + \kappa^{-1}MK_2^2 \right) <1, \
\nu^{-1}\kappa^{-1}Ri^2C_P^4 <1.
\end{eqnarray}
\end{lemma}
The following stronger condition will also be used in the sequel in order to 
simplify some of the constants.
Define $\eta := \min\{ \nu, \kappa\}$,
then by the definition of $K_1,K_2$ in \eqref{bd1}, 
we have the following inequality 
\begin{align*}
\nu^{-1}M\left( 2K_1 + \kappa^{-1}MK_2^2 \right) & \le \eta^{-2}M \left( 2(RiC_P^2\kappa^{-1}\|\gamma\|_{-1} + \|f\|_{-1}) + \kappa^{-2}M\|\gamma\|_{-1}^2 \right) ,\\
\nu^{-1}\kappa^{-1}Ri^2C_P^4& \le \eta^{-2}Ri^2C_P^4.
\end{align*}
Then a stronger sufficient condition for uniqueness of solutions to  
\eqref{b1a}-\eqref{b2a} is given by 
\begin{eqnarray}
\eta^{-2}M \left( 2(RiC_P^2\kappa^{-1}\|\gamma\|_{-1} + \|f\|_{-1}) + \kappa^{-2}M\|\gamma\|_{-1}^2 \right)<1,  \quad \eta^{-1}RiC_P^2<1,
\label{sdcineq}
\end{eqnarray}
which implies  $\kappa^{-1}RiC_P^2<1$. 
\begin{proof}
Assume $(u,\theta), (w,z)\in (V_h, W_h)$ are solutions to the \eqref{b1a}-\eqref{b2a}.  Subtracting these two systems produces
\begin{align*}
b(u,u-w,v) + b(u-w,w,v) + \nu(\nabla (u-w), \nabla v)  & = Ri(\langle 0, \theta-z\rangle^T, v) ,\\
b^*(u,\theta- z, \chi) + b^*(u-w, z,\chi) + \kappa (\nabla (\theta- z), \nabla \chi) & = 0.
\end{align*}
Setting $v = u-w, \chi = \theta - z$ eliminates the first nonlinear terms in both equations and yields
\begin{align*}
\nu\|\nabla (u-w)\|^2 &\le  \nu^{-1}Ri^2 C_P^4 \|\nabla (\theta-z)\|^2 + 2MK_1\|\nabla (u-w)\|^2,\\ 
\kappa \|\nabla (\theta-z)\|^2 & \le  \kappa^{-1}M^2 K_2^2 \|\nabla (u-w)\|^2,
\end{align*}
thanks to the Cauchy-Schwarz and  Poincar\'e inequalities, together with
\eqref{bbd} and \eqref{bd1}.
Combining these bounds, we obtain
\begin{eqnarray}\label{eqn:sdcond}
\left( 1-2\nu^{-1}MK_1 - \nu^{-1}\kappa^{-1}M^2K_2^2 \right) \nu \|\nabla (u-w)\|^2 + \left( 1-\nu^{-1}\kappa^{-1}Ri^2 C_P^4 \right)\kappa\|\nabla (\theta- z)\|^2 \le 0.
\end{eqnarray}
Under the conditions \eqref{sdcineq1},
both terms on the left-hand side of \eqref{eqn:sdcond} are nonnegative, 
and in fact positive unless $u = w$ and $\theta=z$, implying the solution is 
unique.
\end{proof}

In the remainder, we will assume either
\eqref{sdcineq1} or \eqref{sdcineq} holds, to guarantee the well-posedness of the
system \eqref{b1a}-\eqref{b2a}. 
For notational simplicity, we prefer using the stronger \eqref{sdcineq}. 
All of the results in section \ref{sec:decoupled} hold as well for 
\eqref{sdcineq1} with minor differences in the constants.

\subsection{Anderson acceleration}
The extrapolation technique known as Anderson acceleration, which is used to improve 
the convergence of a fixed-point iteration, 
may be stated as follows \cite{ToKe15,WaNi11}.
Consider a fixed-point operator $g: Y \goto Y$ where $Y$ is a normed vector space. 

\begin{alg}[Anderson iteration] \label{alg:anderson}
Anderson acceleration with depth $m \ge 0$ and damping factors $0 < \beta_k \le 1$.\\ 
Step 0: Choose $x_0\in Y.$\\
Step 1: Find $w_1\in Y $ such that $w_1 = g(x_0) - x_0$.  
Set $x_1 = x_0 + w_1$. \\
Step $k$: For $k=2,3,\ldots$ Set $m_k = \min\{ k-1, m\}.$\\
\indent [a.] Find $w_{k} = g(x_{k-1})-x_{k-1}$. \\
\indent [b.] Solve the minimization problem for $\{ \alpha_{j}^{k}\}_{k-m_k}^{k-1}$
\begin{align}\label{eqn:opt-v0}
\textstyle \min 
	\left\| \left(1- \sum\limits_{j=k-m_k}^{k-1} \alpha_j^{k} \right) w_k + \sum\limits_{j = k-m_k}^{k-1} w_{j} \right\|_Y.
\end{align}
\indent [c.] For damping factor $0 < \beta_k \le 1$, set
\begin{align}\label{eqn:update-v0}
\textstyle
x_{k} 
= (1-\sum\limits_{j = k-m_k}^{k-1}\alpha_j^k) x_{k-1} + \sum_{j= k-m_k}^{k-1} \alpha_j^{k} x_{j-1}
+ \beta_k \left(  (1- \sum\limits_{j= k-m_k}^{k-1} \alpha_j^{k}) w_k + \sum\limits_{j=k-m_k}^{k-1}\alpha_j^k w_{j}\right),
\end{align}
where $w_{j} = g(x_{j-1}) - x_{j-1}$
may be referred to as the update step or as the nonlinear residual.
\end{alg}
Depth $m=0$ returns the original fixed-point iteration.
For purposes of implementation with depth $m > 0$, 
it makes sense to write the algorithm in terms
of an unconstrained optimization problem rather than a constrained problem as 
in \eqref{eqn:opt-v0} \cite{FaSa09,PR19a,WaNi11}.
Define matrices $E_k$ and  $F_k$, whose columns are the consecutive differences between
iterates and residuals, respectively.
\begin{align}
\label{eqn:Ekdef}
E_{k-1} & := \left( \begin{array}{cccccc}e_{k-1} & e_{k-2} &  
  \cdots & e_{k-m_k} \end{array} \right), \quad e_j = x_j - x_{j-1},
\\ \label{eqn:Fkdef}
F_k & := \left( \begin{array}{cccccc}(w_{k}-w_{k-1}) & (w_{k-1} - w_{k-2}) &  
  \cdots & (w_{k-m_k+1} - w_{k-m_k}) \end{array} \right).
\end{align}
Then defining 
$\gamma^{k} = \argmin_{\gamma \in \mathbb{R}^m} \left\|w_{k}-F_k \gamma\right\|_Y$,
the update step \eqref{eqn:update-v0} may be written as
\begin{align}\label{eqn:update-v1}
x_{k} = x_{k-1} + \beta_k w_{k} - (E_{k-1}  +\beta_k F_k) \gamma^{k} 
= \xa_{k-1} + \beta_k \wa_{k}, 
\end{align}
where $\wa_{k} = w_{k} - F_k \gamma^{k}$ and $\xa_{k-1} = x_{k-1} - E_{k-1} \gamma^{k}$,
are the averages corresponding to the solution from the optimization problem. 
The optimization gain factor $\xi_{k}$ may be defined by 
\begin{align}\label{eqn:thetadef}
\nr{\wa_{k}}= \xi_{k}\nrs{w_{k}}.
\end{align}
The gain factor $\xi_k$ plays a critical role in the recent theory \cite{EPRX20,PR19a} 
that shows how this acceleration technique improves convergence.  
Specifically, the acceleration reduces the contribution from the first-order residual 
term by a factor of $\xi_k$, but introduces higher-order terms into the residual 
expansion of the accelerated iterate.

The next two assumptions, summarized from \cite{PR19a}, give sufficient conditions
on the fixed point operator $g$, for the analysis presented there to hold.
\begin{assumption}\label{assume:g}
Assume $g\in C^1(Y)$ has a fixed point $x^\ast$ in $Y$,
and there are positive constants $C_0$ and $C_1$ with
\begin{enumerate}
\item $\nr{g'(x)} \le C_0$ for all $x\in Y$, and 
\item $\nr{g'(x) - g'(y)} \le C_1 \nr{x-y}$
for all $x,y \in Y$.
\end{enumerate}
\end{assumption}

\begin{assumption}\label{assume:fg} 
Assume there is a constant $\sigma> 0$ for which the differences between consecutive
residuals and iterates satisfy
$$ \| w_{{k}+1} - w_{k}\|_Y  \ge \sigma \| x_{k} - x_{{k}-1} \|_Y, \quad {k} \ge 1.$$
\end{assumption}

Under Assumptions \ref{assume:g} and \ref{assume:fg}, 
the following result summarized from 
\cite{PR19a} produces a one-step bound on the residual $\nr{w_{k+1}}$ in terms
of the previous residual $\nr{w_k}$.

\begin{theorem} [Pollock, Rebholz, 2019] \label{thm:genm}
Let Assumptions \ref{assume:g} and \ref{assume:fg} hold,  and suppose the direction 
sines between columns of $F_j$ defined by \eqref{eqn:Fkdef} 
are bounded below by a constant $c_s>0$, for $j = k-m_k, \ldots, k-1$.
Then the residual $w_{k+1} = g(x_k)-x_k$ from Algorithm \ref{alg:anderson} 
(depth $m$) satisfies the following bound.
\begin{align}\label{eqn:genm}
\nr{w_{k+1}} & \le \nr{w_k} \Bigg(
 \xi_k ((1-\beta_{k}) + C_0 \beta_{k})
+ \f{C C_1\sqrt{1-\theta_k^2}}{2}\bigg(
 \nr{w_{k}}h(\xi_{k})
\nonumber \\ &
 + 2  \sum_{n = k-{m_{k}}+1}^{k-1} (k-n)\nr{w_n}h(\xi_n) 
 + m_{k}\nr{w_{k-m_{k}}}h(\xi_{k-m_{k}})
 \bigg) \Bigg),
\end{align}
where  each $h(\xi_j) \le C \sqrt{1 - \xi_j^2} + \beta_{j}\xi_j$,
and $C$ depends on $c_s$ and the implied upper bound on the direction cosines.  
\end{theorem}

The one-step estimate \eqref{eqn:genm} shows how the relative contributions from 
the lower and higher order terms are determined by the gain $\xi_k$ 
from the optimization problem.
The lower order terms are scaled by $\xi_k$ and the higher-order terms are scaled
by $\sqrt{1 - \xi_k^2}$.  
Greater algorithmic depths $m$ generally give smaller values of $\xi_k$ as the 
optimization is run over an expanded search space. 
However, the reduction comes at the cost of both increased accumulation and weight 
of higher order terms.
If recent residuals are small, then this may be negligible and greater algorithmic
depths $m$ may be advantageous (to a point). 
If the previous residual terms are large however (e.g., near the beginning of 
an iteration), then greater depths $m$ may slow or prevent convergence 
in many cases.

\section{A decoupled fixed point iteration for the Boussinesq system}
\label{sec:decoupled}
All results in this section hold for both Dirichlet boundary conditions \eqref{dbc}, 
and mixed boundary conditions \eqref{mbc}, with minor differences in the constants. 
We will use the Dirichlet boundary conditions for illustration. 
One can easily extend the analysis for the mixed boundary conditions.

We will consider the following fixed point iteration:  
Given $u_0,\theta_0$, for $k=1,2,3,...$, 
find $(u_k,p_k,\theta_k)\in (X_h,Q_h,W_h)$ satisfying
for all $v\in X_h, q\in Q_h, \chi\in W_h$, 
\begin{align}
b(u_{k-1}, u_k , v) + \nu (\nabla u_k, \nabla v) - (p_k, \nabla \cdot v)  
&= Ri\left(  \langle 0, \theta_{k}\rangle^T , v\right) + (f,v),\\
(\nabla \cdot u_k, q) 
& = 0,\\
b^*(u_{k-1}, \theta_k, \chi) +\kappa(\nabla \theta_k, \nabla \chi) 
&= (\gamma, \chi).
\end{align}
For finite element spaces $X_h, Q_h$ satisfying the discrete inf-sup condition as
described in subsection \ref{subsec:mathprelim}, 
we have the equivalent formulation in the discretely divergence-free space 
\eqref{eqn:ddfs}: 
for any $v\in V_h, \chi\in W_h$,
\begin{align} \label{it1}
b(u_{k-1}, u_k , v) + \nu (\nabla u_k, \nabla v)  & = Ri\left(  \langle 0, \theta_{k}\rangle^T , v\right) + (f,v),\\
b^*(u_{k-1}, \theta_k, \chi) +\kappa(\nabla \theta_k, \nabla \chi) &= (\gamma, \chi). 
\label{it2}
\end{align} 
The following subsections develop a framework which will allow us to analyze this 
iteration.

\subsection{A solution operator $G$ corresponding to the fixed point iteration}
We will consider the solution operator of the system \eqref{it1}-\eqref{it2} as the 
fixed-point operator defining the iteration to be accelerated.  
To study this operator, we next formally define it in a slightly more abstract way.

Given $f\in H^{-1}(\Omega)^d$, $\gamma\in H^{-1}(\Omega)$, and $(u,\theta)\in (V_h,W_h)$, consider the problem of finding $(\tilde u,\tilde \theta)\in  (V_h,W_h)$ satisfying
\begin{align} \label{it1b}
b(u, \tilde u , v) + \nu (\nabla \tilde u, \nabla v)  
& = Ri\left(  \langle 0,\tilde \theta \rangle^T , v\right) + (f,v),\\
b^*(u, \tilde \theta, \chi) +\kappa(\nabla \tilde \theta, \nabla \chi) 
&= (\gamma, \chi), \label{it2b}
\end{align} 
 for any $v\in V_h, \chi\in W_h$.

\begin{lemma} \label{wellp1}
For $f\in H^{-1}(\Omega)^d$ and $\gamma\in H^{-1}(\Omega)$, the system  \eqref{it1b}-\eqref{it2b} is well-posed, and solutions satisfy the bounds
\begin{align}
\| \nabla \tilde u \|  \le  K_1  \text{ \quad and \quad } \| \nabla \tilde \theta \|   \le  K_2, \label{bound2}
\end{align}
where $K_1, \ K_2$ are given in \eqref{bd1}.
\end{lemma}
\begin{proof}
We begin with {\em a priori} bounds.  Suppose solutions exist, and choose 
$\chi=\tilde \theta$ and $v=\tilde u$.  By construction, this vanishes the trilinear 
terms in each equation. After applying Cauchy-Schwarz, Poincar\'e and H\"older 
inequalities, it can be seen that
\begin{align*}
\nu \| \nabla \tilde u \|^2 & \le C_P^2 Ri \| \nabla \tilde\theta \| \| \nabla \tilde u \| + \| f \|_{-1} \| \nabla \tilde u \|, \\
\kappa \| \nabla \tilde \theta \|^2 & \le \| \gamma \|_{-1} \| \nabla \tilde \theta \|.
\end{align*}
The second bound reduces to 
$\| \nabla \tilde \theta \|  \le \kappa^{-1} \| \gamma \|_{-1}$,
and inserting this into the first bound produces
\begin{align*}
\| \nabla \tilde u \| & \le C_P^2 Ri \nu^{-1}   \kappa^{-1} \| \gamma \|_{-1} + \nu^{-1} \| f \|_{-1} .
\end{align*}
Since the system \eqref{it1b}-\eqref{it2b} is linear in $\tilde u$ and $\tilde \theta$,
and finite dimensional, 
these bounds are sufficient to imply solution uniqueness and therefore existence.
\end{proof}

\begin{dfn}\label{defg00}
Define $G: (V_h,W_h) \to (V_h,W_h)$ to be the solution operator of \eqref{it1b}-\eqref{it2b}. 
That is,
\[
(\tilde u,\tilde \theta)=G(u,\theta) .
\]
\end{dfn}
By Lemma \ref{wellp1}, $G$ is well defined.  
The Boussinesq fixed point iteration \eqref{it1}-\eqref{it2} can now be written as 
\[
(u_k,\theta_k)=G(u_{k-1},\theta_{k-1}).
\]
Before we give a norm on $(V_h,W_h)$, we recall scalar multiplication 
on the ordered pair $(u,\theta) \in (V_h,W_h)$ satisfies
$ \alpha \cdot (u,\theta) = (\alpha u, \alpha \theta),$ 
for any $\alpha \in \mathbb{R}$. 
\begin{dfn}
Define the norm $\| (\cdot,\cdot)\|_B: (V_h,W_h) \rightarrow \mathbb{R} $ by
\[
\|(v,w)\|_B := \sqrt{ \nu \| \nabla v\|^2 + \kappa \| \nabla w\|^2}.
\]
\end{dfn}
The weights used in the norm definition come from the natural energy norm of the Boussinesq system, and this norm will be referred to as the B-norm.  
Using this weighted norm both
simplifies the analysis and improves the practical implementation, 
as this norm will be used in the optimization step of the accelerated algorithm.

\subsection{Continuity and Lipschitz differentiability of $G$}

We now prove that $G$ satisfies Assumptions \ref{assume:g} and \ref{assume:fg}.  We begin with Lipschitz continuity.

\begin{lemma}\label{contract}  There exists a positive constant $C_G$ such that $\| G(u,\theta) - G(w,z)\|_B \le C_G \| (u,\theta) - (w,z)\|_B$ for any $(u,\theta), (w,z) \in (V_h,W_h).$ The constant $C_G$ is defined by
\[ 
C_G = \nu^{-1/2}\eta^{-1/2}M \sqrt{2K_1^2 + 3K_2^2 }.
\]
\end{lemma}

\begin{proof}For any $(u,\theta) \in (V_h,W_h)$, denote $( G_1(u,\theta), G_2(u,\theta))$ as the components of $G(u,\theta)$.
Let $(u,\theta),(w,z) \in (V_h,W_h)$.  Set $G(u,\theta) = (G_1(u,\theta), G_2(u,\theta))$ and $G(w,z) = (G_1(w,z), G_2(w,z))$.  Then for all $(v,\chi) \in (V_h,W_h)$ we have
\begin{align}
\label{fg1a}b(u,G_1(u,\theta), v) + \nu(\nabla G_1(u,\theta), \nabla v) &= Ri (\langle 0, G_2(u, \theta) \rangle^T, v) + (f,v) ,\\
\label{fg2a} b^*(u,G_2(u,\theta), \chi) + \kappa (\nabla G_2(u,\theta), \nabla \chi) & = (\gamma, \chi),\\
\label{fg3a} b(w, G_1(w,z), v) + \nu (\nabla G_1(w,z), \nabla v) &= Ri(\langle 0,G_2(w,z)\rangle^T, v) + (f,v) ,\\
\label{fg4a} b^*(w,G_2(w,z), \chi) + \kappa (\nabla G_2(w,z), \nabla \chi)  &= (\gamma , \chi).
\end{align}
Subtracting 
\eqref{fg3a}-\eqref{fg4a} from
\eqref{fg1a}-\eqref{fg2a} gives
\begin{align}
\label{fg5}
b(u,G_1(u,\theta) -  G_1(w,z),v) + b(u-w,G_1( w,z),v) 
 &+ \nu\left(\nabla (G_1( u,\theta) - G_1( w,z)), \nabla v\right) 
\nonumber \\ =
& Ri (\langle 0, G_2(u,\theta)- G_2(w,z) \rangle^T , v), 
\\ 
b^*(u, G_2(u,\theta) - G_2(w,z),\chi)+ b^*(u-w,G_2(w,z),\chi)  
  &+ \kappa \left( \nabla (G_2(u, \theta) - G_2(w, z)),
\nabla \chi \right) = 0, 
\label{fg6}
\end{align}
Choosing $\chi = G_2(u,\theta) - G_2(w,z)$ in \eqref{fg6} eliminates the first term, 
then applying \eqref{bbd} and \eqref{bound2} gives
\begin{eqnarray}
\label{bdg2}
 \|\nabla (G_2(u,\theta) - G_2(w,z))\| \le \kappa^{-1}M K_2 \| \nabla (u-w)\|.
\end{eqnarray}
Similarly, choosing $v = G_1( u,\theta) - G_1(w,z)$ in \eqref{fg5} eliminates the 
first nonlinear term, and applying \eqref{bbd}, H\"older and Poincar\'e inequalities 
produces
\begin{align}
\label{bdg1}
 \|\nabla (G_1( u,\theta) - G_1(w,z))\|
&\le  \nu^{-1}C_P^2 Ri \| \nabla (G_2(u,\theta) - G_2(w,z))\|   + \nu^{-1}M \| \nabla (u-w)\|\|\nabla G_1( w,z)\|  \nonumber \\
&\le \nu^{-1} M(\kappa^{-1}K_2RiC_P^2  + K_1)\|\nabla (u-w)\| \nonumber \\
&\le  \nu^{-1}M(K_1+K_2) \|\nabla (u-w)\|,
 \end{align}
thanks to \eqref{bound2} of Lemma \ref{wellp1} and \eqref{sdcineq}.
Combining \eqref{bdg2}-\eqref{bdg1} gives
 \[
\| G(u,\theta) - G( w, z) \|_B \le  \nu^{-1/2}\eta^{-1/2}M \sqrt{2K_1^2 + 3K_2^2 }
 \| (u,\theta)-(w,z)\|_B.
\]
Thus $G$ is Lipschitz continuous with constant 
{$C_G = \nu^{-1/2}\eta^{-1/2}M \sqrt{2K_1^2 + 3K_2^2 }$.}
\end{proof}

Next, we show that $G$ is Lipschitz Fr\'echet differentiable. 
We will first define a mapping $G'$,
and then in Lemma \ref{LFD} confirm that it is the Fr\'echet derivative operator of $G$.  
\begin{dfn}\label{def:gprime}
Given $(u,\theta)\in (V_h,W_h)$, define an operator $G'(u,\theta;\cdot,\cdot): (V_h,W_h) \to (V_h,W_h)$ by 
\[
G'(u,\theta;h,s) := \left(G_1'(u,\theta;h,s), G_2'(u,\theta;h,s)\right),
\]
 satisfying  for all $(h,s), (v,\chi)\in (V_h,W_h)$,
\begin{align}
b(h,G_1(u,\theta), v) + b(u, G_1'(u,\theta;h,s), v) + \nu (\nabla G_1'(u,\theta;h,s), \nabla v)  &= 
  Ri (\langle 0, G_2'(u,\theta; h,s)\rangle^T,v), \nonumber\\
   \label{defg1} \\
b^*(h, G_2(u,\theta), \chi) + b^*(u ,G_2'(u,\theta;h,s), \chi)  
+ \kappa(\nabla G_2'(u,\theta;h,s), \nabla \chi) & = 0. \label{defg2}
\end{align}
\end{dfn}
Once it is established that $G'$ is well-defined and is the Fr\'echet derivative of 
$G$, it follows that $G'(u,\theta; \cdot, \cdot)$ is the Jacobian matrix of $G$ at 
$(u,\theta)$. 
From the partially decoupled system \eqref{defg1}-\eqref{defg2}, it is clear that 
$G'(u,\theta; \cdot,\cdot)$ is a block upper triangular matrix. 
Applied to any $(h,s) \in (V_h,W_h)$, the resulting $G'(u,\theta; h,s)$ 
can be written componentwise as
$\left(G_1'(u,\theta;h,s), G_2'(u,\theta;h,s)\right) \in (V_h, W_h).$
\begin{lemma} \label{LFD}
The Boussinesq operator $G$ is Lipschitz Fr\'echet differentiable: there exists a constant $ \hat C_G$ such that for all $(u,\theta),\ (w,z), \ (h,s) \in (V_h,W_h)$
\begin{align}\label{eqn:lfd1} 
\| G'(u,\theta; h,s) \|_B \le C_G\|(h,s)\|_B ,
\end{align}
and 
\begin{align}\label{eqn:lfd2}
\| G'(u+h ,\theta+s;w,z) - G'(u,\theta;w,z) \|_B \le \hat C_G \|(h,s) \|_B\| (w,z)\|_B,
\end{align}
where $ C_G$ is defined in Lemma \ref{contract}.
\end{lemma}

\begin{proof}
The first part of the proof shows that  $G$ is Fr\'echet differentiable, and
\eqref{eqn:lfd1} holds. We begin by finding an upper bound on the norm of $G'$ and 
then showing $G'$ is well-defined. 
Setting $\chi = G_2'(u,\theta; h,s)$ in \eqref{defg2}  eliminates 
the second nonlinear term and  gives 
\begin{equation}
\label{bddg2}
\| \nabla G_2'(u,\theta; h,s)\| \le  \kappa^{-1}M \| \nabla h\| \| \nabla G_2(u+h,\theta+s)\| \le \kappa^{-1}MK_2\|\nabla h\|,
\end{equation}
thanks to Lemma \ref{wellp1}. Similarly, setting $v = G_1'(u,\theta; h,s)$ 
in \eqref{defg1} eliminates the second nonlinear term and  yields 
\begin{align}
\label{bddg1}
 \| \nabla G_1'(u,\theta; h,s) \| 
&\le  \nu^{-1}RiC_P^2 \|\nabla G_2'(u,\theta;h,s)\|  + \nu^{-1}M \| \nabla h\| \|\nabla G_1(u,\theta)\|  \nonumber  \\
 &\le \nu^{-1}M \left(  \kappa^{-1}  K_2RiC_P^2 +K_1 \right)\| \nabla h\| \nonumber\\
 &\le \nu^{-1}M(K_1+K_2 ) \|\nabla h\|,
\end{align}
thanks to Lemma \ref{wellp1}, \eqref{bddg2} and the small data condition 
\eqref{sdcineq}.
Combining the bounds \eqref{bddg2}-\eqref{bddg1} yields
\begin{eqnarray}
\| G'(u,\theta; h,s) \|_B \le  C_G \|(h,s)\|_B.
\label{fdbd}
\end{eqnarray}
Since system \eqref{defg1}-\eqref{defg2} is linear and finite dimensional, \eqref{fdbd} is sufficient to imply the system is well-posed.  
Therefore, $G'$ is well-defined and uniformly bounded over $(V_h,W_h),$ since the 
bound is independent of $(u,\theta)$.

Next, we prove $G'$ given by definition \ref{def:gprime} is the Fr\'echet derivative 
operator of $G$. That is, given $(u,\theta)\in (V_h,W_h)$, there exists some constant $C$ such that for any $(h,s)\in (V_h,W_h)$ $$\| G(u+h,\theta+s) - G(u,\theta) - G'(u,\theta; h,s)\|_B \le C\|(h,s)\|^2 .$$
For notational ease, set $\tilde g_1 = G_1(u+h,\theta+s) - G_1(u,\theta) - G_1'(u,\theta;h,s)$ and $\tilde g_2 = G_2(u+h,\theta+s) - G_2(u,\theta) - G_2'(u,\theta;h,s)$.
To construct the left hand side of the inequality above, we begin with the following equations: 
for any $(u,\theta), (h,s), (v,\chi) \in (V_h,W_h)$,
\begin{align}
b(u+h, G_1(u+h,\theta+s), v) + \nu(\nabla G_1(u+h,\theta+s), \nabla v) &=  
Ri(\langle 0, G_2(u+h,\theta+s)\rangle^T ,v) + (f,v),  \nonumber \\  & \label{eqnd1}\\
b^*(u+h, G_2(u+h,\theta+s), \chi) + \kappa (\nabla G_2(u+h,\theta+s), \nabla \chi) 
&= (\gamma, \chi),\label{eqnd2}\\
b(u,G_1(u,\theta), v) + \nu(\nabla G_1(u,\theta), \nabla v) 
&=   Ri(\langle 0, G_2(u,\theta)\rangle^T, v) + (f,v) ,  \label{eqnd3}\\
b^*(u,G_2(u,\theta), \chi) + \kappa (\nabla G_2(u,\theta) , \nabla \chi)&= (\gamma, \chi). \label{eqnd4}
\end{align}
Subtracting \eqref{eqnd1}-\eqref{eqnd2} from \eqref{eqnd3}-\eqref{eqnd4} 
and \eqref{defg1}-\eqref{defg2}, and then choosing $v = \tilde g_1, \chi = \tilde g_2$, we obtain by
application of \eqref{bbd}, H\"older and Sobolev inequalities 
\cite[Chapter 6]{laytonbook} that
\begin{align*}
 \nu\| \nabla \tilde g_1 \|^2 
& =   -b(h,G_1(u+h,\theta + s) - G_1(u,s),  \tilde g_1 )  + Ri(\langle 0, \tilde g_2\rangle^T, \tilde g_1)\\
&\le  M \| \nabla h\| \|\nabla (G_1(u+h,\theta+s) - G_1(u,\theta)\| \| \nabla \tilde g_1 \|+ RiC_P^2\|\nabla \tilde g_2\| \|\nabla \tilde g_1\|,\\
\kappa \|\nabla \tilde g_2\|^2 
&= -b(h,G_2(u+h,\theta+s) - G_2(u,\theta), \tilde g_2),
\end{align*}
which reduces to 
\begin{align}
\label{eqn5}
\kappa \|\nabla \tilde g_2\|^2 &\le \kappa^{-1}M^2 \|\nabla h\|^2\|\nabla (G_2(u+h,\theta+s) - G_2(u,\theta))\|^2,
\end{align}
and
\begin{align}\label{eqn6}
 & \nu \|\nabla \tilde g_1\|^2   \nonumber\\
 &\le    2\nu^{-1} M^2\| \nabla h\|^2 
\big(  \|\nabla (G_1(u+h,\theta+s) - G_1(u,\theta)\|^2  
\nonumber \\
&+ Ri^2C_P^4 \kappa^{-2}\|\nabla (G_2(u+h,\theta+s) - G_2(u,\theta))\|^2 \big) 
\nonumber \\
&\le  2\eta^{-1}\nu^{-1} M^2\| \nabla h\|^2 \|\nabla  (G(u+h,\theta+s) - G(u,\theta))\|_B^2 , 
\end{align}
thanks to Young's inequality and \eqref{sdcineq}.
Combining bounds \eqref{eqn5}-\eqref{eqn6} produces
\begin{align*}
  \| (\tilde g_1,\tilde g_2) \|_B^2 &
\le 3\eta^{-2}M^2 \|\nabla h\|^2 \|\nabla (G(u+h,\theta+s) - G(u,\theta))\|_B^2 \le 
3\eta^{-3}M^2C_G^2 \|(h,s)\|_B^4.
\end{align*}
By the definitions of $\tilde g_1,\tilde g_2$, this shows
\begin{eqnarray}
\| G(u+h,\theta+s) - G(u,\theta) - G'(u,\theta;h,s) \|_B   \le
2\eta^{-3/2}MC_G\|(h,s)\|_B^2,
\label{fdiff}
\end{eqnarray}
which demonstrates Fr\'echet differentiability of $G$ at $(u,\theta)$.
As \eqref{fdiff} holds for arbitrary $(u,\theta)$, 
we have that $G$ is Fr\'echet differentiable on all of $(V_h,W_h)$.

The second part of the proof shows that $G'$ is Lipschitz continuous over $(V_h,W_h)$.  By the definition of $G'$, the following equations hold 
\begin{align} 
\label{eqnld1}
b(w,G_1(u,\theta),v) + b(u, G_1'(u,\theta;w,z), v) + \nu (\nabla G_1'(u,\theta;w,z), \nabla v)  & =    \nonumber \\
Ri(\langle 0,G_2'(u,\theta; w,z)\rangle^T ,v), &\\ 
\label{eqnld2}
b^*(w,G_2(u,\theta), \chi) + b^*(u,G_2'(u,\theta;w,z), \chi) 
+ \kappa (\nabla G_2'(u,\theta;w,z), \nabla \chi) & = 0,\\
b(w,G_1(u+h,\theta+s),v) + b(u+h, G_1'(u+h,\theta+s;w,z), v) & \nonumber \\
+ \nu (\nabla G_1'(u+h,\theta+s;w,z), \nabla v)   =  
Ri(\langle 0,G_2'(u+h,\theta+s; w,z)\rangle^T ,v), &\label{eqnld3} \\
b^*( w,G_2(u+h,\theta+s), \chi) 
+ b^*(u+h,G_2'(u+h,\theta+s;w,z), \chi) & \nonumber \\
+ \kappa (\nabla G_2'(u+h,\theta+s;w,z), \nabla \chi) & = 0,
\label{eqnld4}
\end{align}
for all $(u,\theta), (w,z), (h,s), (v,\chi)\in (V_h,W_h)$.  Letting  $e_1 := G_1'(u+h,\theta+s;w,z) - G_1'(u,\theta;w,z),\ e_2 := G_2'(u+h,\theta+s;w,z) - G_2'(u,\theta;w,z)$ and 
subtracting \eqref{eqnld1}-\eqref{eqnld2} from
\eqref{eqnld3}-\eqref{eqnld4} 
gives
\begin{align*}
b(w,G_1(u+h, \theta+s) - G_1(u,\theta), v) +b(h,G_1'(u,\theta; w,z), v) + b(u+h, e_1, v) + \nu(\nabla e_1, \nabla v) &=   \\
 Ri (\langle 0, e_2\rangle^T,  v),& \\
 b^*(w, G_2(u+h,\theta+s) - G_2(u,\theta), \chi) + b^*(h,G_2'(u,\theta;w,z),\chi) + b^*(u+h,e_2,\chi)
 + \kappa (\nabla e_2, \nabla v) & = 0.
\end{align*}
Setting $v = e_1 $, $\chi = e_2 $ eliminates the last nonlinear terms in both equations and produces
\begin{align*}
\nu\|\nabla e_1\|^2 & \le RiC_P^2\|\nabla e_2\|\|\nabla e_1\| + M\|\nabla w\|\|\nabla (G_1(u+h, \theta+s) - G_1(u,\theta))\| \|\nabla e_1\| \\
& \ \ \ + M\|\nabla h\| \|\nabla G_1'(u,\theta ;w,z)\|\|\nabla e_1\|,\\
\kappa \|\nabla e_2\|^2 &\le M\|\nabla w\| \|\nabla (G_2(u+h,\theta+s) - G_2(u,\theta))\| \|\nabla e_2\| + M \|\nabla h\| \|\nabla G_2'(u,\theta;w,z)\|\|\nabla e_2\| ,
\end{align*}
thanks to \eqref{bbd}.
Thus from \eqref{bdg2}-\eqref{bdg1}, \eqref{bddg2}-\eqref{bddg1}, 
 \eqref{fdbd} and  Lemma \ref{contract}, we obtain
\begin{align*}
 \|\nabla e_2\| & \le 2M^2 \kappa^{-2}K_2\|\nabla w\| \|\nabla h\|, \\
\|\nabla e_1\| & \le \nu^{-1}RiC_P^2 \|\nabla e_2\| + 2\nu^{-2}M^2(K_1+K_2)\|\nabla h\|\nabla w\| \\
&\le 2\nu^{-1}\eta^{-1}M^2 (K_1+2K_2)\|\nabla w\|\|\nabla h\|.
\end{align*}
Combing these bounds gives
\begin{eqnarray*}
\|(e_1, e_2)\|_B^2 \le  4\nu^{-2} \eta^{-3}M^4 (2K_1^2 + 9K_2^2) \|(h,s)\|_B^2\|(w,z)\|_B^2 = \hat C_G^2\|(h,s)\|_B^2\|(w,z)\|_B^2
\end{eqnarray*}
where $\hat C_G =2\nu^{-1} \eta^{-3/2}M^2 (2K_1^2 + 9K_2^2)^{1/2}$ is the Lipschitz constant of $G'$.
Thus $G'(u,\theta;\cdot,\cdot)$ is Lipschitz continuous with constant $\hat C_G$.  
As the bound holds for arbitrary $(u,\theta)$, we have that $G$ is Lipschitz 
continuously differentiable on $(V_h,W_h)$ with constant $\hat C_G$.
\end{proof}

{It remains to show that Assumption \ref{assume:fg} is satisfied for the solution operator $G$.  This amounts to
finding a constant $\sigma>0 $ such that  for any $(u,\theta), (w,z)\in (V_h, W_h)$ with $\|\nabla u\| , \|\nabla w\|\le K_1$ 
\begin{equation}\label{siglem}
\| F(u,\theta) - F(w,z) \|_B \ge \sigma \| (u,\theta) - (w,z)\|_B, 
\end{equation}
where $F(u,\theta) : = G(u,\theta) - (u,\theta).$
\begin{lemma}
\label{lemma2} Assume the problem data is such that $G$ is contractive, i.e. $C_G<1$.  
Then there exists a constant $\sigma>0 $ such that \eqref{siglem} holds for any $(u,\theta), (w,z)\in (V_h, W_h)$ with $\|\nabla u\| , \|\nabla w\|\le K_1.$
\end{lemma}
\begin{proof}
Since $G$ is contractive, it follows from {\it Remark 4.2} in \cite{PR19a} that \eqref{siglem} holds with $\sigma=1-C_G$.
\end{proof}

This shows 
that Assumption 2.4 is satisfied under the data restriction that $C_G<1$, which is similar but not equivalent
to the uniqueness conditions in section 2.  However, in our numerical tests, which used data far larger than these restrictions, 
the $\sigma_k$'s calculated at each iteration (i.e. \eqref{siglem} with $u=u_k,\theta=\theta_k,w=u_{k-1},z=\theta_{k-1}$) were in general no smaller
than $10^{-3}$.  Hence it may be possible to prove that Assumption \ref{assume:fg} holds under less restrictive data restrictions.
}

\subsection{Accelerating the decoupled Boussinesq iteration}\label{subsec:accelbouss}

In this section, we provide the algorithm of Anderson acceleration applied to the 
decoupled Boussinesq system \eqref{it1} - \eqref{it2} with either Dirichlet \eqref{dbc} or mixed 
boundary conditions \eqref{mbc}.
The one-step residual bound is stated below for Boussinesq solve operators $G$.
\begin{alg}[Anderson accelerated iterative method for Boussinesq equations] 
\label{algaa}
\noindent
\begin{itemize} 
\item[Step 0] Give an initial $u_0\in X_h, \theta_0 \in W_h$.
\item[Step 1a)] Find $\tilde \theta_1\in W_h$ satisfying for all $\chi \in W_h$ 
\begin{eqnarray*}
b^*(u_{0}, \tilde \theta_1, \chi) +\kappa(\nabla \tilde\theta_1, \nabla \chi) &=& (\gamma, \chi).
\end{eqnarray*}
\item[Step 1b)] Find  $ \tilde u_1\in V_h$ satisfying for all $v\in V_h$
\begin{eqnarray*}
b(u_{0}, \tilde u_1, v) + \nu (\nabla \tilde u_1, \nabla v)
& =& Ri\left(  \langle 0, \tilde \theta_1\rangle^T, v\right) + (f,v).
\end{eqnarray*}
Then set $u_1 = \tilde u_1, \theta_1 = \tilde \theta_1$, and $w_1 = \tilde u_1 - u_0, \ z_1 = \tilde \theta_1 - \theta_0.$
\item[Step k] For $k = 2,3,\dots$, set $m_k = \min\{k-1 ,m\}$.
\begin{itemize}
\item[[a.]] Find $\tilde \theta_{k}\in W_h$ by solving 
\begin{eqnarray}\label{eqnaa2}
b^*( u_{k-1},\tilde \theta_{k}, \chi) +\kappa(\nabla \tilde \theta_{k}, \nabla \chi) &=& (\gamma, \chi) ,
\end{eqnarray}
\item[[b.]]  Find $\tilde u_{k} \in V_h$ by solving
\begin{eqnarray} \label{eqnaa1}
b(u_{k-1}, \tilde u_{k} , v) + \nu (\nabla \tilde u_{k}, \nabla v)  
& =& Ri\left(  \langle 0, \tilde  \theta_{k}\rangle^T , v\right) + (f,v),
\end{eqnarray}
and then compute $w_{k} = \tilde u_{k} - u_{k-1}, z_{k} = \tilde \theta_{k} - \theta_{k-1}$.
\item[[c.]] Solve the minimization problem \begin{align}
\min\| (w_k^\alpha, z_k^\alpha)\|_B
\end{align} for $\{\alpha_j^{k}\}_{j= k -m_k}^{k-1}$
where $(w_k^\alpha, z_k^\alpha) \coloneqq  \left(1- \sum\limits_{j=k-m_k}^{k-1} \alpha_j^{k}\right)(w_{k}, z_{k}) +  \sum\limits_{j = k-m_k}^{k-1} \alpha_j^{k+1} (w_{j}, z_{j})  .$
\item[[d.]] For damping factor $0< \beta_k \le 1$, set 
\begin{multline*}
(u_{k}, \theta_{k}) =   \left( 1-\sum\limits_{j=k-m_k}^{k-1}\alpha_j^{k}\right)(u_{k-1},\theta_{k-1}) +  \sum\limits_{j = k-m_k}^{k-1} \alpha_j^{k} (u_{j-1},\theta_{j-1})  \\
+ \beta_k\left(\left(1- \sum\limits_{j=k-m_k}^{k-1} \alpha_j^{k}\right)(w_{k}, z_{k}) +  \sum\limits_{j = k-m_k}^{k-1} \alpha_j^{k+1} (w_{j}, z_{j}) \right).
\end{multline*} 
\end{itemize}
\end{itemize}
\end{alg}

As the Boussinesq operator $G$ satisfies Assumptions \ref{assume:g} and 
\ref{assume:fg}, we have the following convergence result of Algorithm \ref{algaa} 
by directly applying {\it Theorem 5.5}  from \cite{PR19a}.

\begin{thm}[Convergence result for Algorithm \ref{algaa}]
\label{thmaa}
The $k^{th}$ step residual $(w_{k}, z_{k})$ generated from Algorithm \ref{algaa} satisfies 
\begin{multline*}
\|(w_{k+1}, z_{k+1})\|_B \le \| (w_{k}, z_{k})\|_B \Bigg(  \xi_{k}(1-\beta_{k} + \hat K_g\beta_{k} ) + \frac{C \hat C_G\sqrt{1-\xi_k^2}}{2} \bigg( \|(w_{k},z_{k})\|_B h(\xi_{k}) 
\\
 + 2\sum\limits_{j= k- m_{k}+1}^{k-1} (k-j)\|(w_j, z_j)\|_B h(\xi_j) + m_{k-1}\|(w_{k-m_{k}}, z_{k-m_{k}})\|_B h(\xi_{k-m_{k}})   \bigg)  \Bigg),
\end{multline*}
where $\xi_k := \left\| (w_k^\alpha, z_k^\alpha) \right\|_B/\| (w_k, z_k)\|_B,$ $h(\xi_j) \le C \sqrt{1-\xi_j^2} + \beta_{j}\xi_j$, and $C$ depends on a lower bound of the direction sine between $(w_{j+1},z_{j+1})-(w_j,z_j)$ and $span\{(w_{n+1},z_{n+1})  - (w_n,z_n)\}_{n = j-m_k+1}^{j-1}.$
\end{thm}

\section{Numerical experiments for Anderson accelerated Picard iterations}
\label{sec:numerics}

\begin{figure}[ht!]
\begin{center}
\includegraphics[width = .4\textwidth, height=.4\textwidth,viewport=50 0 960 900, clip]{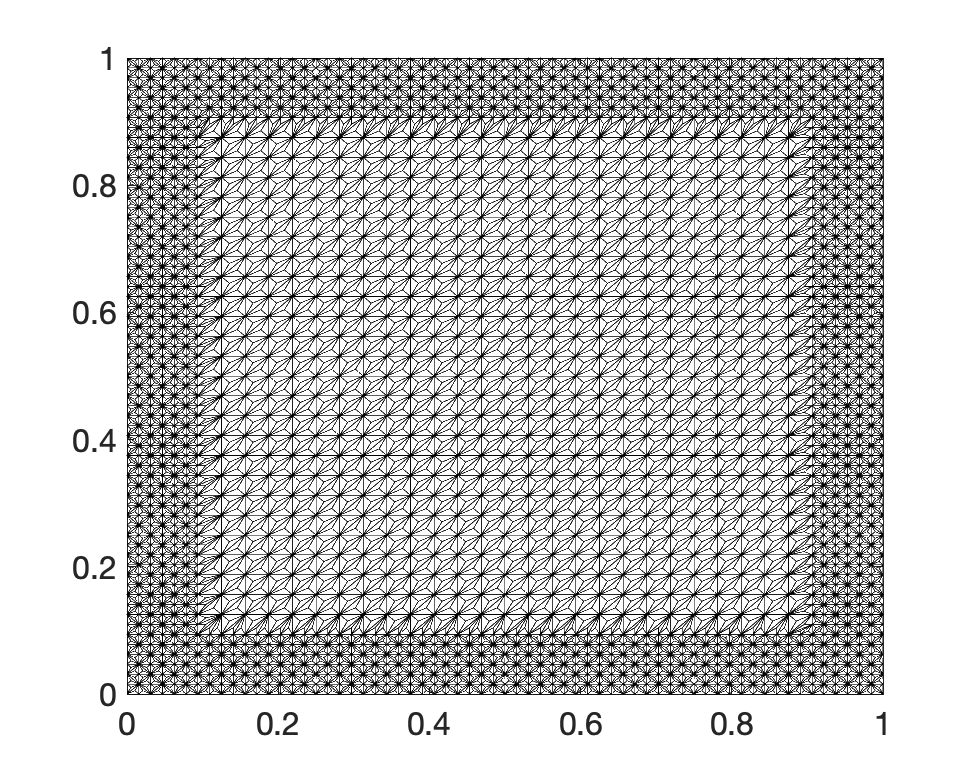}
\caption{\label{mesh}The mesh used for all numerical tests.
}\end{center}
\end{figure}

\begin{figure}[h!]
\vspace*{-20pt}
\begin{center}
\includegraphics[width = .3\textwidth, height=.3\textwidth,viewport=80 180 560 600, clip]{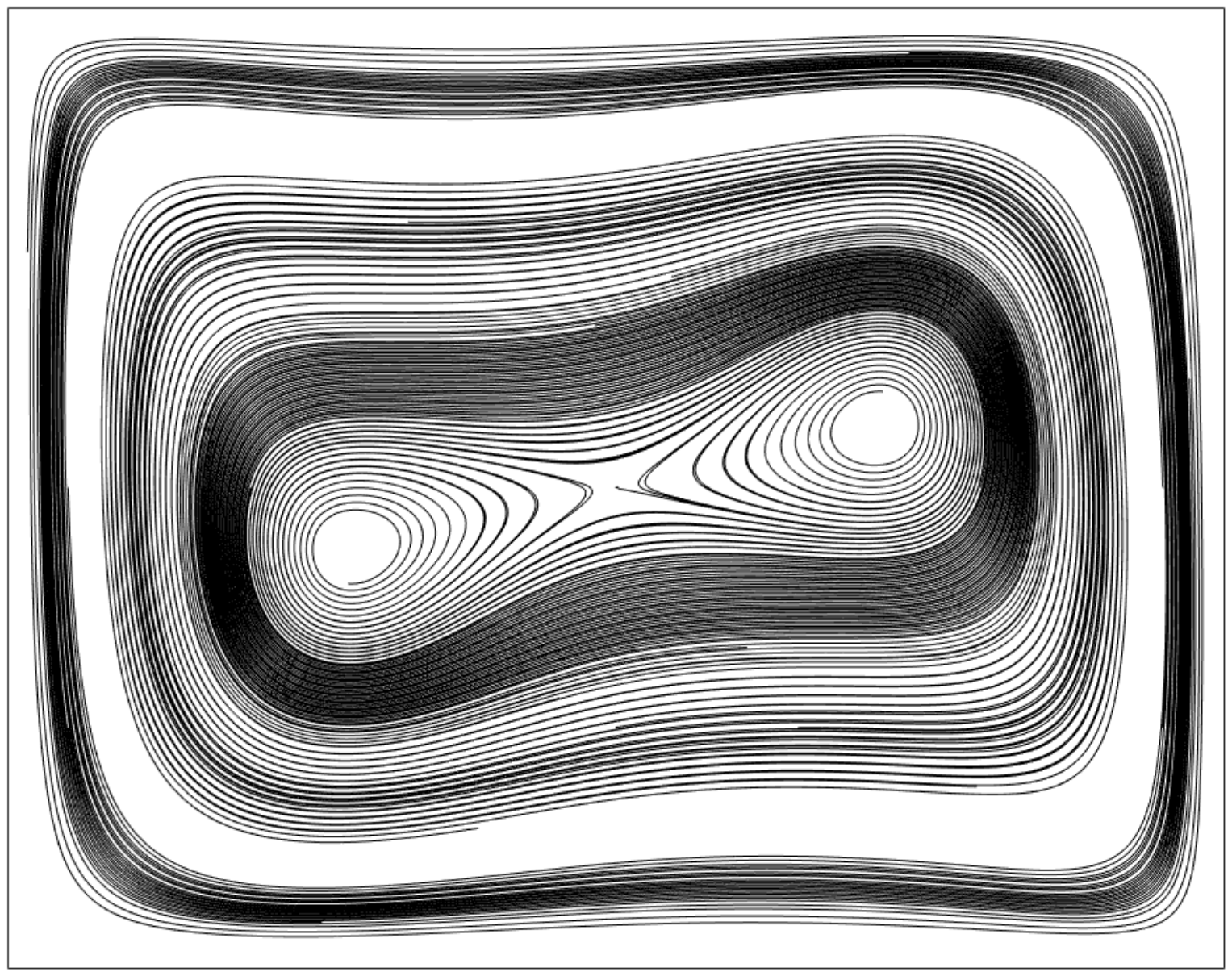}
\includegraphics[width = .3\textwidth, height=.3\textwidth,viewport=80 180 560 600, clip]{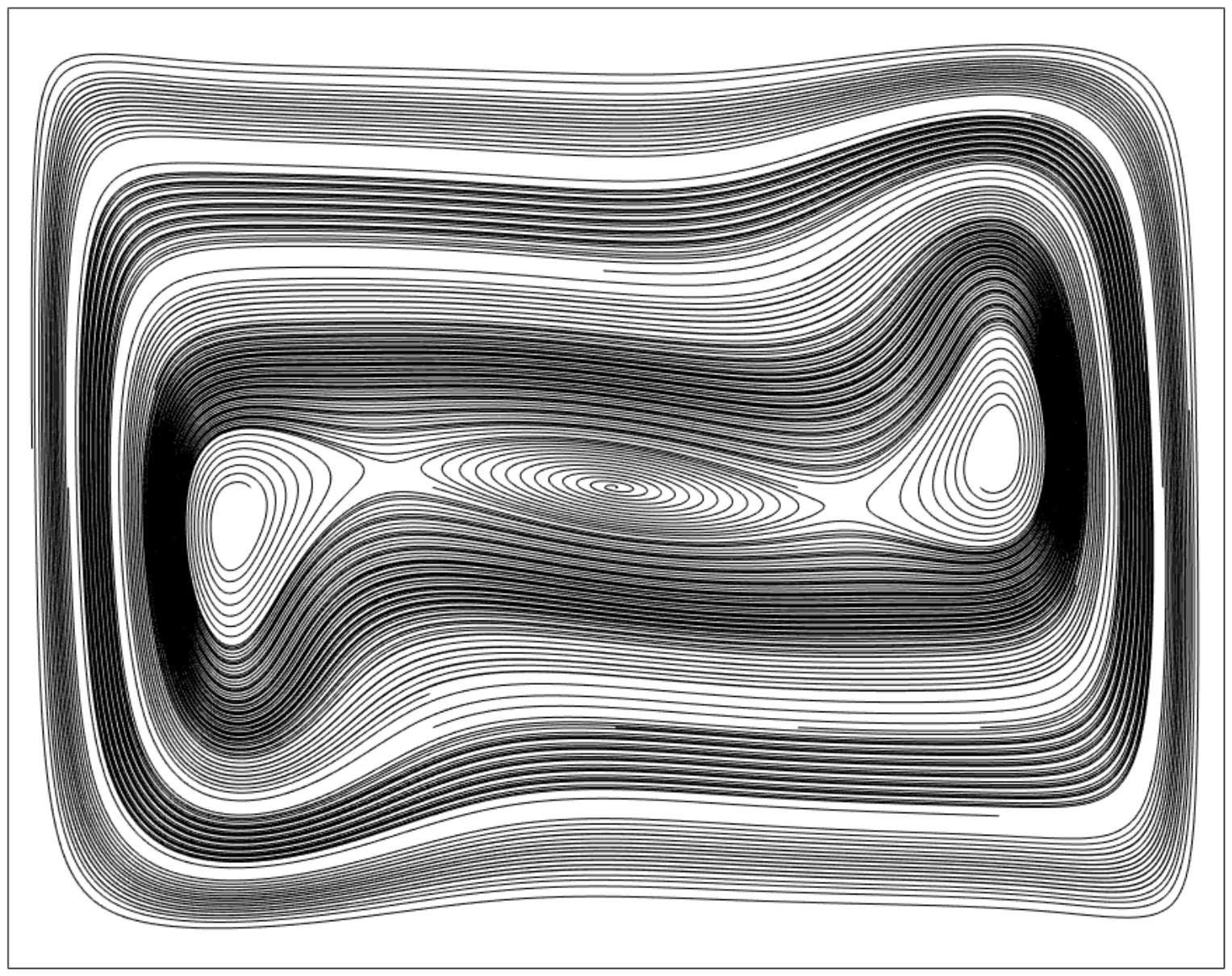}
\includegraphics[width = .3\textwidth, height=.3\textwidth,viewport=80 180 560 600, clip]{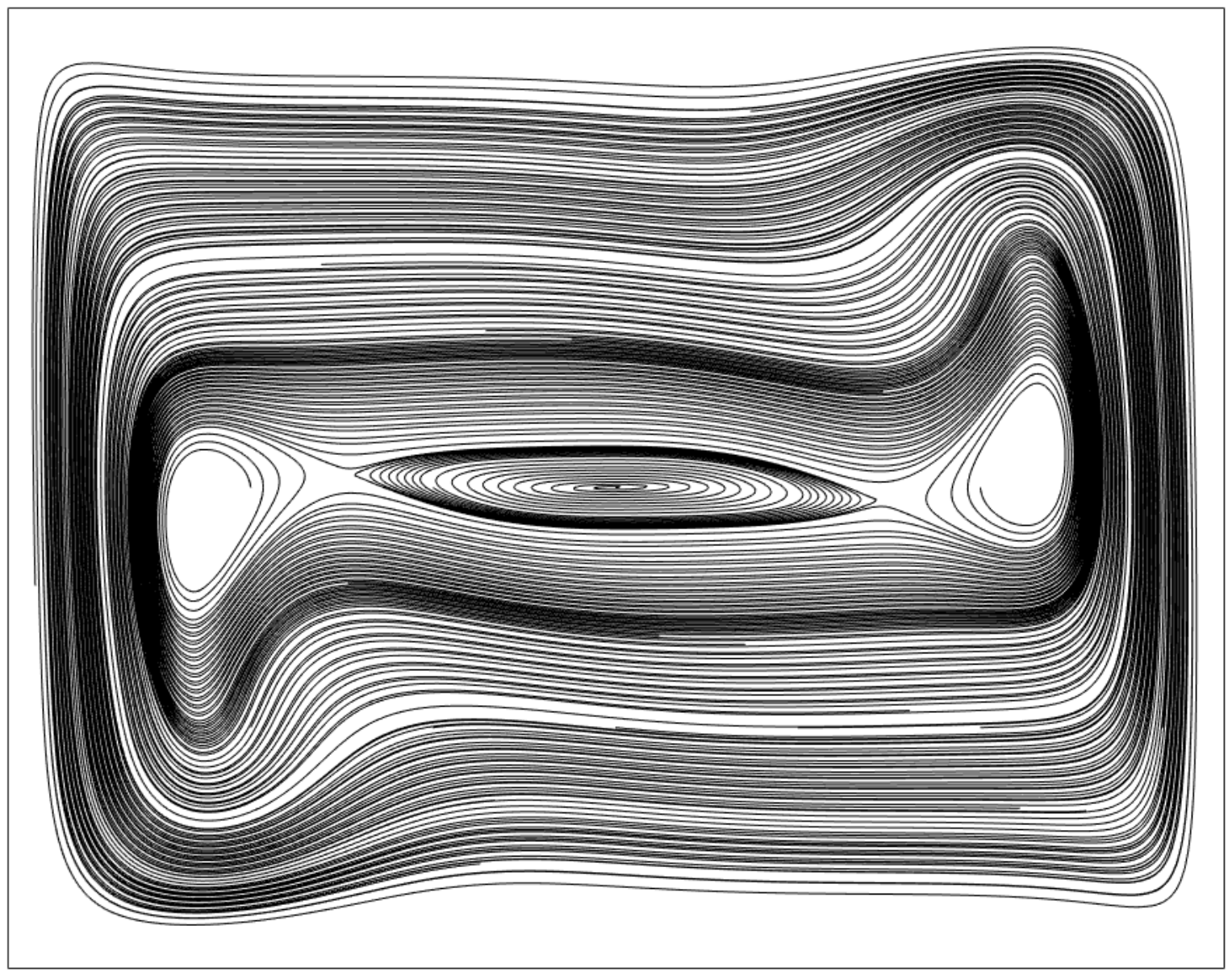}\\
\includegraphics[width = .3\textwidth, height=.3\textwidth,viewport=100 150 520 600, clip]{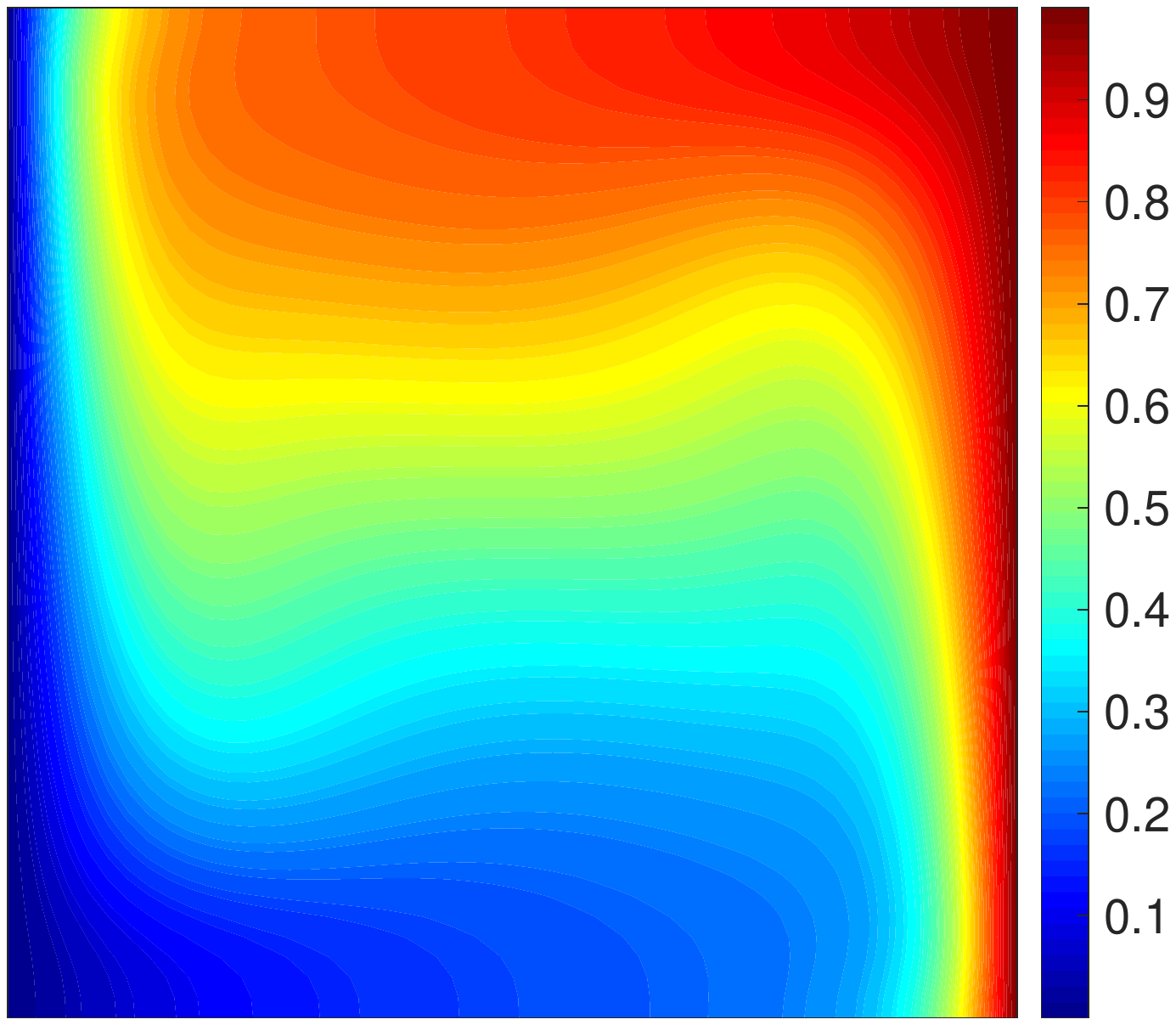}
\includegraphics[width = .3\textwidth, height=.3\textwidth,viewport=100 150 520 600, clip]{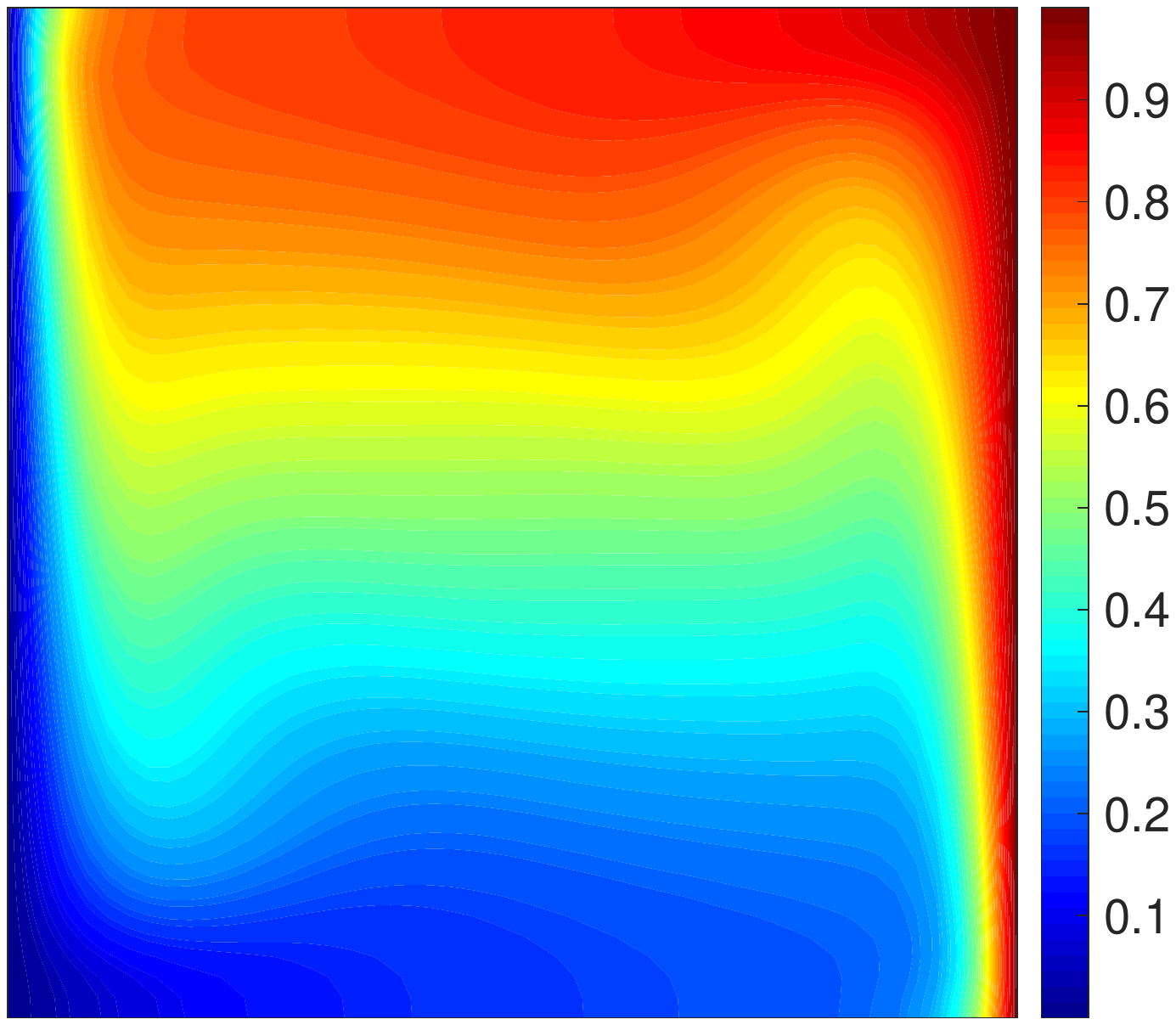}
\includegraphics[width = .3\textwidth, height=.3\textwidth,viewport=100 150 520 600, clip]{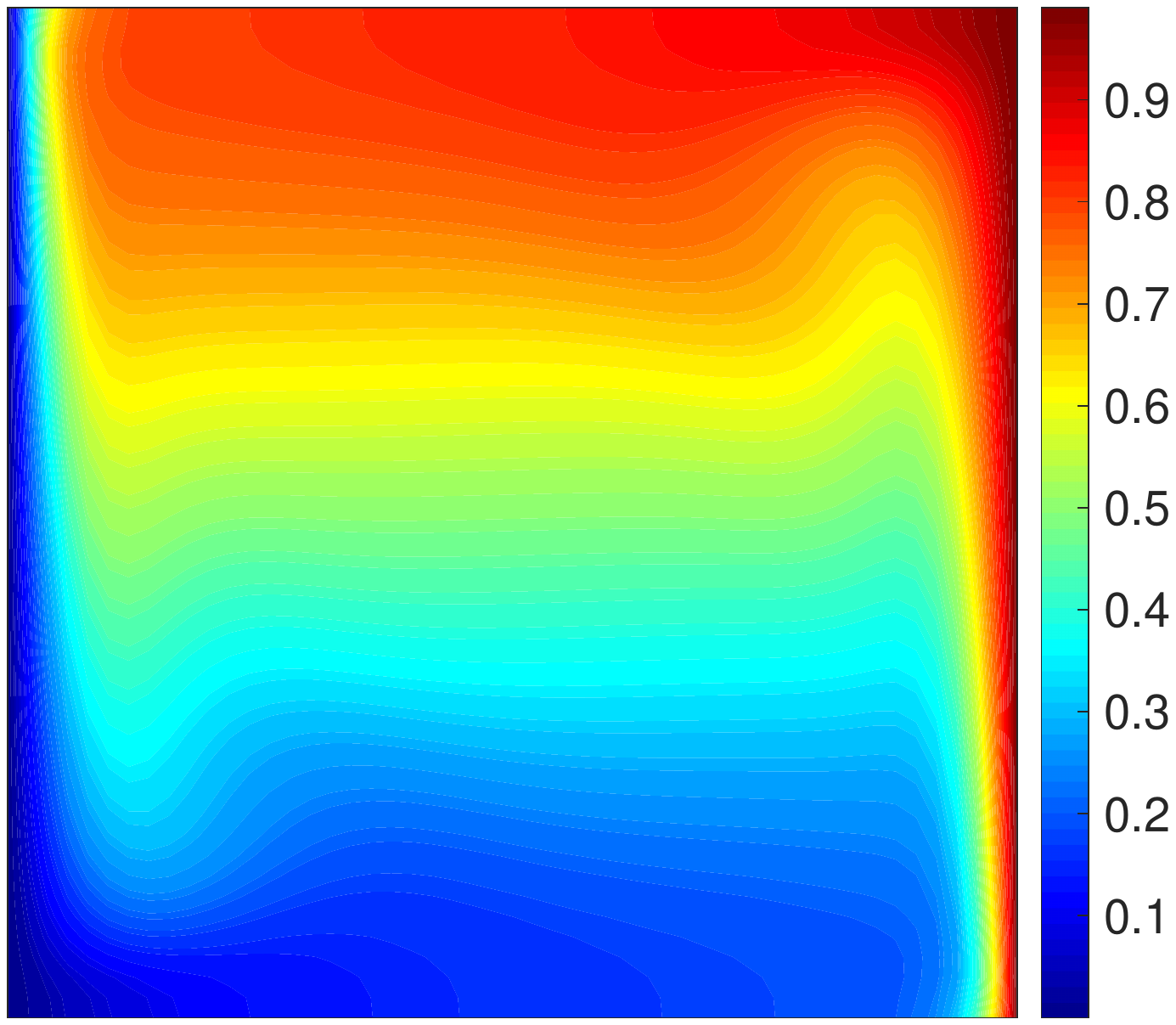}
\vspace*{-20pt}
\caption{\label{Ra5plots}The resolved solution's velocity streamlines (top) and temperature contours (bottom), for $Ra=10^5$, $Ra=5\cdot 10^5$, and $Ra=10^6$, from
left to right.}
\end{center}
\end{figure}

We now demonstrate the {Anderson accelerated Picard iteration for the Boussinesq system} on a benchmark differentially heated cavity problem from \cite{CK11}, with varying Rayleigh number.  The domain for the problem is the unit square, 
and for boundary conditions we enforce no slip velocity on all walls, 
$\nabla T\cdot n=0$ on the top and bottom, and $T(1,y)=1$, $T(0,y)=0$.  
The initial iterates $u_0=0$ and $T_0=0$, are used for all tests.
No continuation method or pseudo time stepping is employed.

The discretization uses $(P_2,P_1^{disc})$ velocity-pressure Scott-Vogelius elements, 
and $P_2$ temperature elements.
The mesh is created by starting with a uniform triangulation, 
refining along the boundary, then applying a barycenter refinement (Alfeld split, in the language of Fu, Guzman and Neilan \cite{FGN20}).  
The resulting mesh is shown in figure \ref{mesh}, and with this element choice provides 89,554 
total degrees of freedom.  We show results below for varying Rayleigh numbers, 
which come from using parameters $\nu=0.01$ and $\kappa=0.01$, and varying $Ri$.  Plots of resolved solutions for varying $Ra$ are shown in figure \ref{Ra5plots}.

In our tests, we consider both constant algorithmic depts $m$, and 
also a 2-stage strategy that uses a smaller $m$ (1 or 2) when the nonlinear residual is 
larger than $10^{-3}$ in the B-norm, and $m=20$ when the residual is smaller.  
The 2-stage approach is motivated by Theorem  \ref{thmaa}, 
which suggests that greater depths $m$ can be detrimental early in the iteration due to 
the accumulation of non-negligible higher order terms in the residual expansion.
Once the residual is sufficiently small, then the reduction of the 
first order terms from a greater depth $m$ can be enjoyed without 
noticeable pollution from the higher-order terms, 
which essentially results in an improved linear convergence rate.

\begin{figure}[ht!]
\begin{center}
$Ra=10^5$ \hspace{2in} $Ra=5\cdot 10^5$ \\
\includegraphics[width = .45\textwidth, height=.3\textwidth,viewport=20 180 550 600, clip]{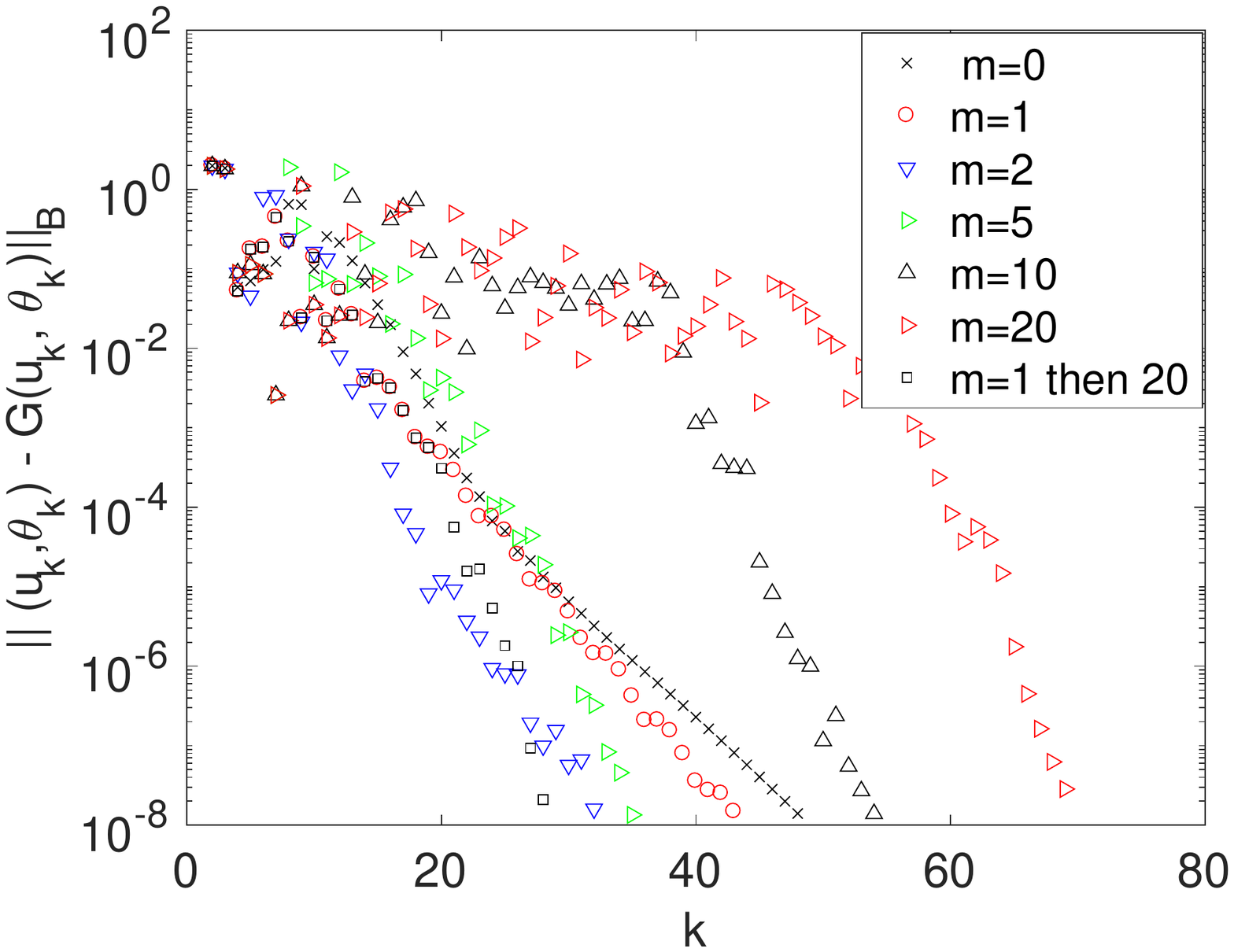}
\includegraphics[width = .45\textwidth, height=.3\textwidth,viewport=20 180 550 600, clip]{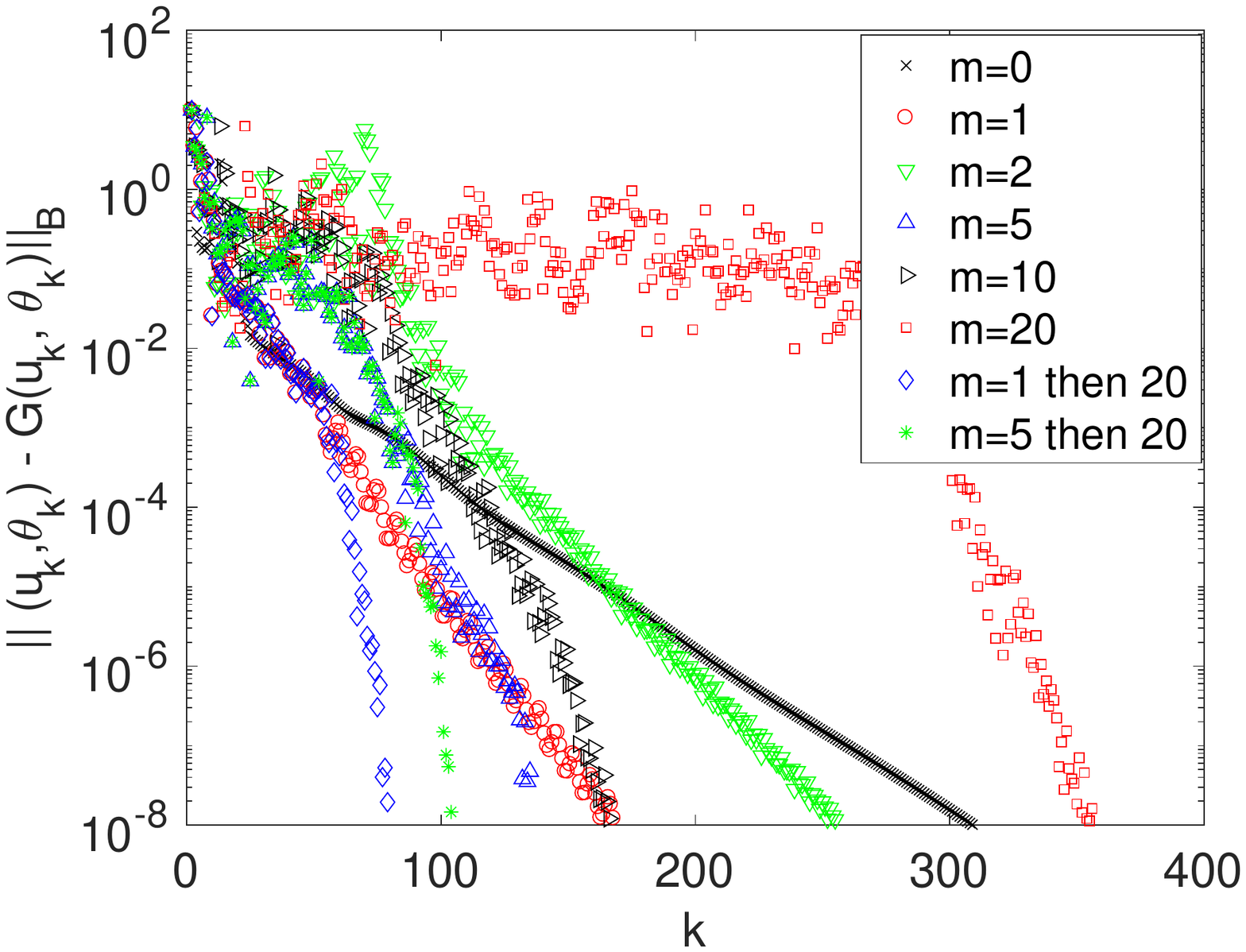} \\
$Ra=10^6$ \hspace{2in} $Ra=2\cdot 10^6$ \\
\includegraphics[width = .45\textwidth, height=.3\textwidth,viewport=20 180 550 600, clip]{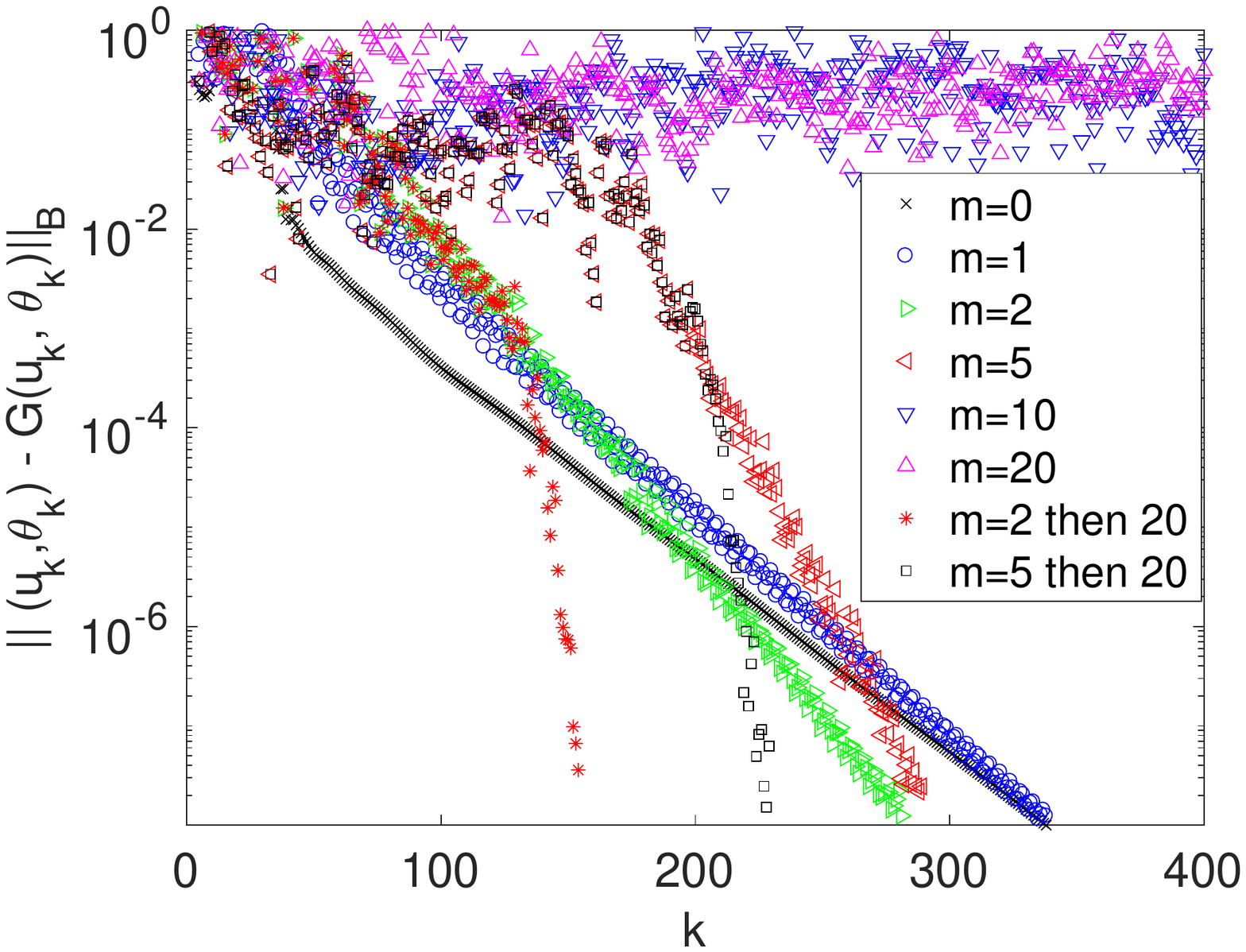}
\includegraphics[width = .45\textwidth, height=.3\textwidth,viewport=20 180 550 600, clip]{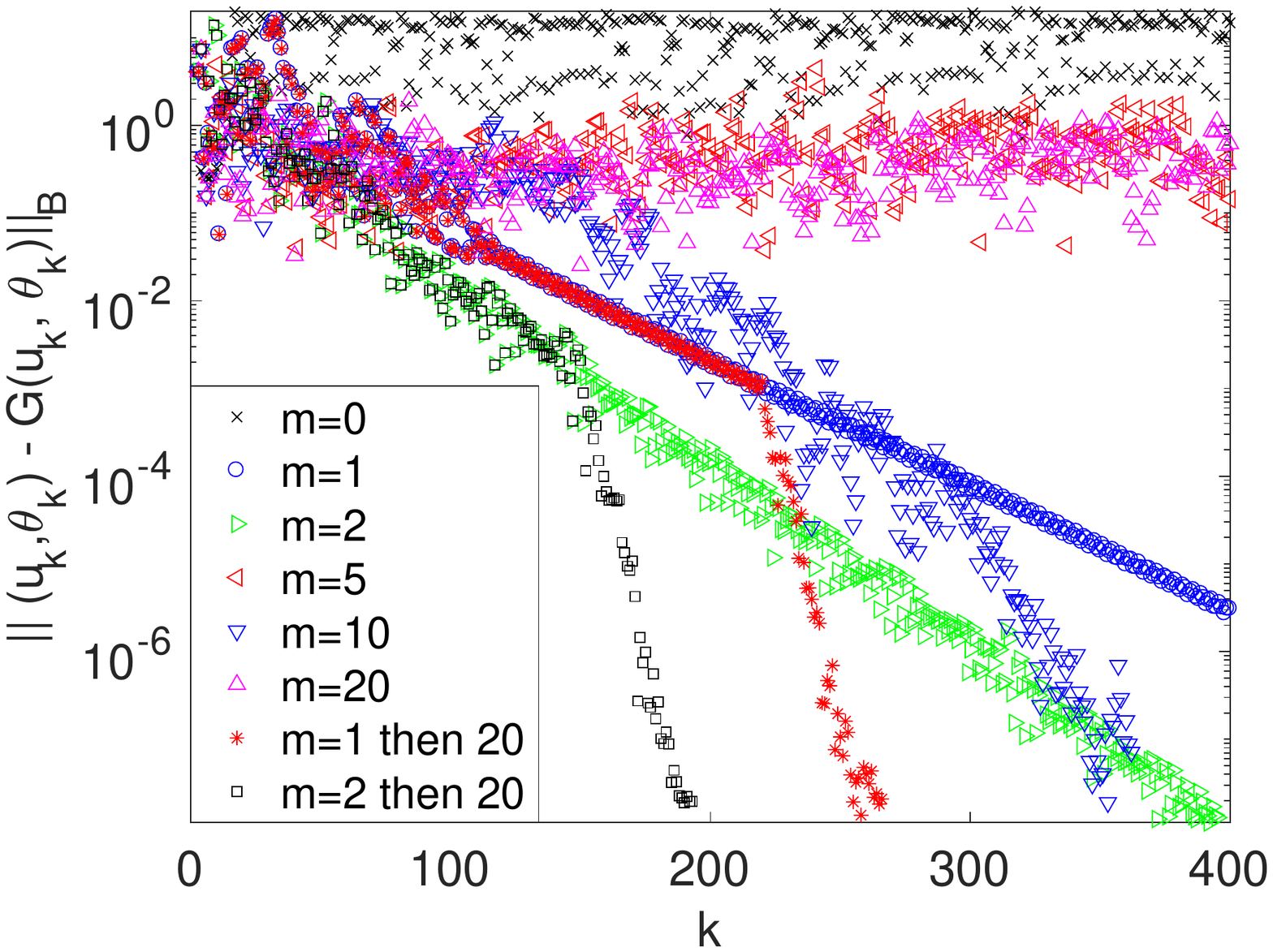}
\caption{\label{Ra5plots2}Convergence results for the differentially heated cavity with varying $m$, for several Rayleigh numbers.}
\end{center}
\end{figure}

We test here Algorithm \ref{algaa}, i.e. the Anderson accelerated Picard iteration for 
the Boussinesq system, for varying choices of depth  $m$ and 
damping parameter $\beta$, and for several
Rayleigh numbers: $Ra=10^5,\ 5\cdot 10^5,\ 10^6,\ 2\cdot 10^6$.  For each $Ra$, we tested with $m=0$ (no acceleration) and each fixed damping parameter 
$\beta=0.05, 0.1, 0.15,\ ..., 1$, and then used the best $\beta$ from 
all tests with that $Ra$. Respectively, for $Ra=10^5,\ 5\cdot 10^5,\ 10^6,\ 2\cdot 10^6$, these parameters  were $\beta=0.3,\ 0.05,\ 0.05,\ 0.05$.  Convergence results for varying $m$ are shown in figure \ref{Ra5plots2}, displayed in terms of the B-norm of the nonlinear residual versus iteration number.  We note that the usual Picard iteration, i.e. $m=0$ and $\beta=1$, did not converge for any of these $Ra$ numbers; after 500 iterations, the B-norm was still larger than $10^{-1}$.  With the appropriately chosen relaxation parameters, the Picard iteration did converge for each case except when $Ra=2\cdot 10^6$.

Convergence results for the lowest Rayleigh number, $Ra=10^5$, are shown in the top left of figure \ref{Ra5plots2} for different values of $m$ and fixed $\beta=0.3$.  Here, we observe the unaccelerated method converges rather quickly, and the accelerated methods that converged faster were $m=1$ and $m=2$ and the 2-stage method that uses $m=1$ and then switches to $m=20$,  with the 2-stage method giving the best results.    Anderson accelerated Picard with $m=5,\ 10,\ 20$ all converged, but slower than if no acceleration was used.  

Results for $Ra=5\cdot 10^5$ also show lower values of $m$ improving convergence but higher values slowing convergence.  Indeed, $m=1, 2, 5$ along with 2-stage methods that used $m=1$ then $20$ and $m=5$ then 20, all outperformed the unaccelerated iteration, while $m=10$ and $m=20$ slowed convergence.  For $Ra=10^6$, the methods with smaller constant $m=1,2,5$ converged, all in roughly the same number of iterations as the unaccelerated method, while the methods run with larger $m=10,\ 20$ did not converge within 400 iterations.  Once again, significant improvement is seen from using 2-stage choices of $m$, and this gave the best results.

With $Ra=2\cdot 10^6$, results showed more improvement from the 
acceleration, as compared to those with lower $Ra$.  
Here, the unaccelerated Picard iteration did not converge, and neither does the iteration with $m=1$. 
Again, the best results come from using a 2-stage choice of $m$, using
a lesser depth at the beginning of the iteration, and a greater depth once the
residual is sufficiently small.

The results described above show a clear advantage from using Anderson acceleration with the Picard iteration, especially for larger Rayleigh numbers.  These results are also in good agreement with our theory

which demonstrates that
Anderson acceleration decreases the first order term, but then
adds higher order terms, to the residual bound.  
Hence using greater algorithmic depths when the residual is large can 
sufficiently pollute the solution so that the improvement found in the reduction of 
the first order term is outweighed by the additional contributions from higher order 
terms.  In all cases, $m=10,\ 20$ slowed convergence compared to $m=0$, 
if the $m=0$ iteration converged. Both theory and these experiments
 also suggest that early in the iteration, moderate choices of algorithmic
depth can be advantageous.  The best results shown here come from the 
2-stage strategy, which takes advantage of the reduction in the first-order
residual term, but only once the residual is small enough that the higher-order 
contributions are negligible in comparison.

\section{Acceleration of the Newton iteration for the Boussinesq system}
\label{sec:newton}

While the theory for this paper focuses on the Picard iteration, it is also of interest to consider the Newton iteration for the Boussinesq system, which takes the form
\begin{align}
(u_{k-1}\cdot \nabla ) u_k  + (u_{k}\cdot \nabla ) u_{k-1}  - (u_{k-1}\cdot \nabla ) u_{k-1} - \nu \Delta u_k + \nabla p_k 
& = Ri \langle 0,\theta_{k} \rangle^T + f, \label{n1a} \\
\nabla \cdot u_k 
&= 0, \label{n2a} \\
(u_{k-1}\cdot \nabla) \theta_k  + (u_{k}\cdot \nabla) \theta_{k-1}  - (u_{k-1}\cdot \nabla) \theta_{k-1}  - \kappa \Delta \theta_k 
& = \gamma, \label{n3a}
\end{align}
together with appropriate boundary conditions.

The Newton iteration is often superior for solving nonlinear problems, 
particularly if one can find 
an initial guess sufficiently close to a solution.  
However, for Boussinesq systems, the Newton
iteration also comes with a significant additional difficulty in that the linear systems that need to be solved at each iteration are fully coupled.  That is, one needs to solve larger block linear systems for $(u_k,p_k,\theta_k)$ simultaneously, whereas for the Picard iteration one first solves for $\theta_k$, and then solves a Navier-Stokes type system for $(u_k,p_k)$.  Hence each iteration of Picard is significantly more efficient than each iteration of Newton.  

We proceed now to test Anderson acceleration applied to the Newton iteration, using the same differentially heated cavity problem and discretization studied above.  
While theory to describe how Anderson acceleration can improve the performance of
Newton iterations has yet to be developed, evidence for the efficacy of the method
has been described in \cite{EPRX20,PS20}.
While the acceleration can be expected to interfere with Newton's quadratic convergence
in the vicinity of a solulution, the advantage explored here is the behavior of the
algorithm outside of that regime.
In the far-field regime, outside the domain of asympotically 
quadratic convergence, the Newton iteration may converge linearly \cite{Ptak75}, or, 
of course, not at all. In comparing acccererated Newton with Newton augmented with
a linesearch, we demonstrate how the acceleration can effectively enlarge the domain
of convergence for Newton iterations. As shown in the experiments below, a damped
accelerated Newton iteration with algorithmic depth $m=10$ can also solve the Boussinesq
system with $Ra = 2 \cdot 10^6$.

In the following tests, convergence was declared if the nonlinear residual 
fell below $10^{-8}$ in the B-norm.  If residuals grew larger than $10^4$ in the B-norm, 
the iteration was terminated, and the test was declared to fail due to (essentially) 
blowup, and is denoted with a `B' in the tables below.  
If an iteration did not converge within 200 iterations but its residuals all stayed below
$10^4$ in the B-norm, we declared it to be a failure and denote it with an `F' 
in the tables below.

We first tested Anderson acceleration applied to the Newton iteration, with varying $m$ and no relaxation ($\beta=1$).   Results are shown in table \ref{newtaa}, and a clear improvement can be seen in convergence for higher $Ra$ as $m$ increases.  For comparison, we also show results of (unaccelerated) Newton with a line search, and give results from two choices of line searches: LS1 continuously cuts the Newton step size ratio in half (up to 1/64), until either finding a step that decreases the nonlinear residual of the finite element problem, or using a step size ratio of 1/64 otherwise.  
LS2 uses `fminbnd' from MATLAB (golden section search and parabolic interpolation) to use the step size ratio from [0.01,1] that minimizes the nonlinear residual of the finite element problem.  While  the line searches do help convergence of the Newton iteration, they do not perform as well as Newton-Anderson ($m=$5 or 10) for higher $Ra$ values.

\begin{table}[ht!]
\centering
\begin{tabular}{|l|c|c|c|c|c|c||c|c|}
\hline
$Ri$ &	$Ra$	&	$m=0$ 	& $m=1$	& $m=2$	& $m=5$	& $m=10$		& $m=0$ + LS1 & $m=0$ + LS2 \\ \hline
1	& 1e+4		&9	&9	&11	&14	&17	& 7 & 7	\\ \hline
10	& 1e+5		&B	&17	&19	&81	&38	& B & 11	\\ \hline
20	& 2e+5		&B	&B	&B	&34	&36	& B & 20	\\ \hline
50	& 5e+5		&B	&B	&B	&44	&56	& B & B	\\ \hline
100	& 1e+6		&B	&B	&B	&B	&B	& B & B	\\ \hline
150	& 1.5e+6	&B	&B	&B	&B	&B	& B & B	\\ \hline
200	& 2e+6	&B	&B	&B	&B	&B	& B & B	\\ \hline
\end{tabular}
\caption{ \label{newtaa} Iterations required for Anderson accelerated Newton iterations to converge, with $\beta=1$ and varying $m$.  Results of the unaccelerated Newton method with line searches are also shown. }
\end{table}

We next considered Anderson acceleration to Newton, but using relaxation of $\beta=0.3$.  Result are shown in table \ref{newtaa2}, and we observe further
improvement compared to the case of $\beta=1$.  Lastly, we considered Anderson acceleration applied to Newton, but choosing the $\beta$ from $\{0.0625, 0.125, 0.25, 0.5,1\}$ that has the smallest residual in the B-norm.  This is essentially a look-ahead step, which increases the cost of each step the accelerated Newton algorithm by a factor of 4.  
Results from choosing $\beta$ this way were significantly better than for constant 
$\beta=1$, and somewhat better than $\beta=0.3$, although not worth the extra cost except when using $\beta=0.3$ failed.  A clear conclusion for this section is that Anderson acceleration with a properly chosen depth and relaxation can significantly improve the ability for the Newton iteration to converge for larger $Ra$.  

\begin{table}[ht!]
\centering
\begin{tabular}{|l|c|c|c|c|c|c|}
\hline
$Ri$ &	$Ra$	&	$m=0$	& $m=1$	& $m=2$	& $m=5$	& $m=10$		\\ \hline
1	& 1e+4			&13	&11	&10	&11	&13		\\ \hline
10	& 1e+5		&B	&18	&18	&23	&52		\\ \hline
20	& 2e+5	&B	&41	&32	&22	&44		\\ \hline
50	& 5e+5		&B	&B	&B	&74	&134		\\ \hline
100	& 1e+6		&B	&B	&B	&95	&F		\\ \hline
150	& 1.5e+6	&B	&B	&B	&B	&141		\\ \hline
200	& 2e+6	&B	&B	&B	&B	&156		\\ \hline
\end{tabular}
\caption{ \label{newtaa2} Iterations required for Anderson acclerated Newton iterations to converge to a nonlinear residual of $10^{-8}$ in the B-norm, with $\beta=0.3$ and varying $m$.  B denotes the nonlinear residual growing above $10^4$, and $F$ denotes no convergence after 200 iterations.}
\end{table}

\begin{table}[ht!]
\centering
\begin{tabular}{|l|c|c|c|c|c|c|}
\hline
$Ri$ &	$Ra$	&	$m=0$	& $m=1$	& $m=2$	& $m=5$	& $m=10$		\\ \hline
1	& 1e+4		&8	&9	&9	&13	&24		\\ \hline
10	& 1e+5	&F	&13	&14	&20	&38		\\ \hline
20	& 2e+5	&123	&21	&27	&32	&61		\\ \hline
50	& 5e+5		&35	&59	&31	&65	&93		\\ \hline
100	&1e+6		&B	&42	&F	&68	&106		\\ \hline
150	& 1.5e+6	&B	&B	&B	&F	&112		\\ \hline
200	& 2e+6	&B	&B	&B	&F	&150		\\ \hline
\end{tabular}
\caption{ \label{newtaa3} Iterations required for Anderson acclerated Newton iterations to converge to a nonlinear residual of $10^{-8}$ in the B-norm, with $\beta$ being chosen as the best from a set of five values, and varying $m$.  B denotes the nonlinear residual growing above $10^4$, and $F$ denotes no convergence after 200 iteraitons.}
\end{table}

\section{Conclusions}\label{sec:conc}
In this paper, we studied Anderson acceleration applied to the Picard iteration for Boussinesq system.  The Picard iteration is advantageous compared to the Newton iteration for this problem because it decouples the linear systems into easier to solve pieces.  Since convergence of Picard iterations is typically slow, it is a good candidate for acceleration.  In this work we showed that the Anderson acceleration analysis framework developed in \cite{PR19a} was applicable to this Picard iteration, by considering each iteration as the application of a particular solution operator, and then proving the solution operator had the required properties laid out in \cite{PR19a}.  This in turn proved
one-step error analysis results from \cite{PR19a} hold for Anderson acceleration applied to the Boussinesq system, and local convergence of the accelerated method
under a small data condition.  
Numerical tests with the 2D differentially heated cavity problem
showed good numerical results demonstrating 
the convergence behavior was consistent with our theory, and in particular that a 2-stage choice of Anderson depth works very well. 
Anderson acceleration applied to the related Newton iteration was also considered in numerical tests and it was found that Anderson acceleration allowed for convergence at significantly higher Rayleigh numbers than the usual Newton iteration with common line search techniques.

\section{Acknowledgements}
The author SP acknowledges support from National Science Foundation through the grant DMS 1852876.

\bibliographystyle{plain}


\end{document}